\documentclass[review,onefignum,onetabnum]{siamart190516}

\usepackage{threeparttable}
\usepackage{graphicx,xcolor}\graphicspath{{figures/}}
\usepackage[T1]{fontenc}
\usepackage{lmodern}
\usepackage{amsmath,amssymb}
\usepackage{amsfonts}
\usepackage{lipsum}

\usepackage{lipsum}
\usepackage{amsfonts}
\usepackage{graphicx}
\usepackage{epstopdf}
\usepackage{algorithmic}
\ifpdf
  \DeclareGraphicsExtensions{.eps,.pdf,.png,.jpg}
\else
  \DeclareGraphicsExtensions{.eps}
\fi

\newsiamremark{remark}{Remark}
\newsiamremark{hypothesis}{Hypothesis}
\crefname{hypothesis}{Hypothesis}{Hypotheses}
\newsiamthm{claim}{Claim}

\headers{Kaczmarz inner-iteration flexible GMRES method}{Y. S. Du, K. Hayami, N. Zheng, K. Morikuni and J. F. Yin}

\title{Kaczmarz-type inner-iteration preconditioned flexible GMRES methods for consistent linear systems}

\author{Yi-Shu Du\thanks{School of Mathematical Sciences, Tongji University, N.O. 1239, Siping Road, Shanghai, 200092, China, and LIP, \'{E}cole Normale Sup\'{e}rieure de Lyon, INRIA, 46 All\'{e}e d'Italie, Lyon, 69364, France 
  (\email{duyishu@tongji.edu.cn}).}
\and Ken Hayami\thanks{National Institute of Informatics, and The Graduate University for Advanced Studies (SOKENDAI), 2-1-2 Hitotsubashi, Chiyoda-ku, Tokyo 100-0003, Japan (\email{hayami@nii.ac.jp}).}
\and Ning Zheng\thanks{Research Center for Statistical Machine Learning, The Institute of Statistical Mathematics, 10-3 Midori-cho, Tachikawa Tokyo 190-8562, Japan (\email{nzheng@ism.ac.jp}).
{{Current address: The Institute of Statistical Mathematics, 10-3 Midori-cho, Tachikawa, Tokyo 190-8562, Japan (\email{nzheng@ism.ac.jp}).}}}
\and Keiichi Morikuni\thanks{Faculty of Engineering, Information and Systems, University of Tsukuba, 1-1-1 Tennodai, Tsukuba, Ibaraki 305-8573, Japan (\email{morikuni@cs.tsukuba.ac.jp}).}
\and Jun-Feng Yin\thanks{School of Mathematical Sciences, Tongji University, N.O. 1239, Siping Road, Shanghai, 200092, China (\email{yinjf@tongji.edu.cn}).}}

\usepackage{amsopn}

\newcommand{\rank}{{\mbox{\rm rank}}}
\usepackage{geometry}
\newcommand\blfootnote[1]{%
\begingroup
\renewcommand\thefootnote{}\footnote{#1}%
\addtocounter{footnote}{-1}%
\endgroup
}

\ifpdf
\hypersetup{
  pdftitle={Kaczmarz-type inner-iteration preconditioned flexible GMRES methods for consistent linear systems},
  pdfauthor={Yi-Shu Du, Ken Hayami, Ning Zheng, Keiichi Morikuni and Jun-Feng Yin}
}
\fi

\begin{document}
\nolinenumbers

\maketitle

\begin{abstract}
	We propose using greedy and randomized Kaczmarz inner-iterations as preconditioners for the right-precondition-\\ed flexible GMRES method to solve consistent linear systems, with a parameter tuning strategy for adjusting the number of inner iterations and the relaxation parameter.
	We also present theoretical justifications of the right-preconditioned flexible GMRES for solving consistent linear systems. Numerical experiments on overdetermined and underdetermined linear systems show that the proposed method is superior to the GMRES method preconditioned by NE-SOR inner iterations in terms of total CPU time. \blfootnote{\textbf{Funding}: This work was funded by JSPS KAKENHI Grant (No.~15K04768 and No.~20K14356), the National Natural Science Foundation of China (No.~11971354) and the China Scholarship Council (No.~201906260146).}
\end{abstract}



\section{Introduction}
Consider solving consistent linear systems
\begin{equation}\label{l1}
A \boldsymbol{x} = \boldsymbol{b}, \quad \boldsymbol{b} \in {\mathcal{R}({A})},
\end{equation}
where $A \in {\mathbb{R}^{m \times n}}$ is not necessarily of full rank and ${\mathcal{R}({A})}$ is the range space of $A$.
In particular, consider the minimum Euclidean-norm solution
\begin{equation}\label{l2}
\operatorname*{min}_{\boldsymbol{x} \in {\mathbb{R}^{n}}} \| \boldsymbol{x} \| _ { 2 } 
\quad \text { s.t. } \quad A \boldsymbol{x} = \boldsymbol{b}, \quad \boldsymbol{b} \in {\mathcal{R}({A})}. 
\end{equation}
The problem~\eqref{l2} is equivalent to the normal equations of the second kind
\begin{equation}\label{n2}
A{A^\mathsf{T}}\boldsymbol{u} = \boldsymbol{b}, \quad  {\boldsymbol{x} = }{A^\mathsf{T}}\boldsymbol{u}, \quad \boldsymbol{b} \in {\mathcal{R}({A})},
\end{equation}
where $(\cdot)^\mathsf{T}$ denotes the transpose.

Direct methods for solving problem~\eqref{l2} or~\eqref{n2} are generally expensive when the coefficient matrix is large and sparse. A well-established iterative method for solving~\eqref{l2} is the (preconditioned) CGNE method~\cite{C55,S03}, which is mathematically equivalent to the (preconditioned) Conjugate Gradient (CG) method~\cite{HS52} applied to~\eqref{n2}. 
Another method is the (preconditioned) MRNE method~\cite{MH15, CMTH19}, which applies the (preconditioned) MINRES method~\cite{PS75} to~\eqref{n2}.   
Note that iterative methods may be slow to converge for ill-conditioned problems since the condition number of $AA^\mathsf{T}$ is the square of that of $A$, and preconditioning becomes necessary.
In~\cite{HYI10}, Hayami, Yin and Ito proposed a 
right-preconditioned generalized minimal residual (GMRES) method called the AB-GMRES method by applying GMRES to $\operatorname*{min}_{\boldsymbol{u} \in {\mathbb{R}^{m}}} \| \boldsymbol{b}-AB\boldsymbol{u} \| _ { 2 }$, where $B \in {\mathbb{R}^{n \times m}}$.


In order to accelerate the convergence of iterative methods and save the storage requirement, inner iterations can be applied as a preconditioner inside the Krylov subspace methods instead of applying preconditioning matrices explicitly. Such techniques are often called inner-outer iteration methods~\cite{S93}. Morikuni and Hayami~\cite{MH13,MH15} proposed a class of inner-iteration Krylov subspace methods by applying stationary inner iterations as implicit preconditioners, and showed their efficiency particularly for ill-conditioned and rank-deficient problems. (See also~\cite{CMTH19}.)

In AB-GMRES, common choices for stationary inner iterations are the 
normal error Gauss-Seidel (NE-GS) and normal error successive overrelaxation (NE-SOR) methods~\cite{BE79,S03}, which are also commonly referred to as the Kaczmarz~\cite{K37} or row action methods~\cite{A84,BL13,HY12,C81,G03,S40}. Since it was proposed in the 1930s, the Kaczmarz method has gained great theoretical development and plentiful practical applications~\cite{B03,C88,E40,E10,PCB12,T71}. Research on the Kaczmarz method was reignited in 2006 and 2009 when Strohmer and Vershymin~\cite{SV06,SV09} proposed the randomized Kaczmarz method with expected exponential rate of convergence. In~\cite{BW18}, Bai and Wu constructed a greedy randomized Kaczmarz method by proposing a more effective probability criterion. In~\cite{P18}, Popa summarized convergence rates for Kaczmarz-type methods, including greedy Kaczmarz~\cite{A84}, randomized Kaczmarz methods~\cite{SV09} and so on. 
For more literature on Kaczmarz-type methods, we refer the reader to~\cite{BWoc18,BWA18,BWO18,BW19}.
Numerical results show that these randomized or greedy Kaczmarz-type methods accelerate the original Kaczmarz method and reduce the required number of iterations and CPU time effectively. Inspired by this randomized framework, we replace the NE-SOR method by the greedy and randomized Kaczmarz methods in AB-GMRES preconditioned by stationary inner iterations.

A motivation for developing such a greedy/randomized inner-iteration preconditioning arises in applications where the operation on a row of a matrix is relatively expensive, such as in basis pursuit problems~\cite{CDM01,CMTH19}. 
We intend to reduce the total number of operations on rows by using greedy/randomized inner iterations instead of NE-SOR inner iterations. We mention related previous work~\cite{AMT10, MS14, RT08} on randomized preconditioners for least squares problems.

When the randomized or greedy Kaczmarz method is used as the inner iteration, the rows of $A$ are selected randomly or greedily in each iteration and thus the preconditioner is not fixed during the outer iteration. Therefore, we use the flexible GMRES method~\cite{S93} as the outer iteration and propose a new algorithm called the flexible AB-GMRES method with Kaczmarz-type inner iterations. 
Theoretically, an optimality property of minimizing residual norm can be given under the framework of flexible GMRES. We also propose a parameter tuning procedure for adjusting the number of inner iterations and the relaxation parameter for the new method. Numerical results show that flexible AB-GMRES preconditioned by Kaczmarz-type methods outperform the AB-GMRES method preconditioned by NE-SOR iterations~\cite{MH15} in terms of total CPU time.

The organization of this paper is as follows. In section 2, we review the AB-GMRES method. In section 3, we present the flexible AB-GMRES method for consistent linear systems, and give an optimality property of the proposed method. In section 4, we propose a parameter tuning procedure for the new method and present numerical experiment results. In section 5, we conclude the paper.

\section{AB-GMRES method}
In this section, the inner-iteration preconditioned AB-GMRES method is briefly introduced.
Consider solving equation~\eqref{l2} using AB-GMRES. AB-GMRES corresponds to GMRES applied to $\mathop {\min }_{\boldsymbol{u} \in {\mathbb{R}^m}} {\left\| {\boldsymbol{b} - AB\boldsymbol{u}} \right\|_2}$ with $\boldsymbol{x} = B\boldsymbol{u}$ and works in an $m$-dimensional space~\cite{HYI10}. 
In order to achieve fast convergence of AB-GMRES and to avoid storing the preconditioner $B$, stationary inner iterations in combination with AB-GMRES were proposed in~\cite{MH15}. This algorithm can be described as follows. Here, $B^{(\ell)}$ denotes the preconditioning matrix for $\ell$ inner iterations.

\begin{algorithm}[H]
	\caption{AB-GMRES method preconditioned by inner iterations~\cite{MH15}}
	\label{ABPI}
	\begin{algorithmic}[1]
		\STATE Let $\boldsymbol{x}_0$ be the initial approximate solution and $\boldsymbol{r}_0=\boldsymbol{b}-A\boldsymbol{x}_0$.
		\STATE $\beta  = {\left\| {{{\boldsymbol{r}}_0}} \right\|_2}$, ${\boldsymbol{v}_1} = {{\boldsymbol{r}}_0}/\beta$
		\FOR{$k = 1,2, \ldots$ until convergence} 
		\STATE Apply $\ell$ iterations of a stationary iterative method to $AA^\mathsf{T}\boldsymbol{y}={\boldsymbol{v}_k}$, $\boldsymbol{z}=A^\mathsf{T}\boldsymbol{y}$ to obtain $\boldsymbol{z}_k=B^{(\ell)}{\boldsymbol{v}_k}$.
		\STATE ${\boldsymbol{w}_k} = A{\boldsymbol{z}_k}$
		\FOR{$i = 1,2, \ldots ,k,$}
		\STATE ${h_{i,k}} = {\boldsymbol{w}_k}^{\mathsf{T}}{\boldsymbol{v}_i}$, ${\boldsymbol{w}_k} = {\boldsymbol{w}_k} - {h_{i,k}}{\boldsymbol{v}_i}$
		\ENDFOR
		\STATE ${h_{k + 1,k}} = {\left\| {{\boldsymbol{w}_k}} \right\|_2}$, ${{\boldsymbol{v}_{k+1}}} = {\boldsymbol{w}_k}/{h_{k + 1,k}}$
		\ENDFOR
		\STATE ${\boldsymbol{y}_k} \!=\! \arg \mathop {\min }_{\boldsymbol{y} \in {\mathbb{R}^k}} {\left\| {\beta {\boldsymbol{e}_1} \!-\! {{\bar H}_k}\boldsymbol{y}} \right\|_2}$, ${\boldsymbol{u}_k} \!=\! \left[ {{\boldsymbol{v}_1},{\boldsymbol{v}_2}, \ldots ,{\boldsymbol{v}_k}} \right]{\boldsymbol{y}_k}$, where ${{\bar H}_k} \!=\! {\left\{ {{h_{ij}}} \right\}_{1 \leq i \leq k + 1;1 \leq j \leq k}}$ and ${\boldsymbol{e}_1}=[1,0,\ldots,0]^\mathsf{T}$
		\STATE Apply $\ell$ iterations of a stationary iterative method to $AA^\mathsf{T}\boldsymbol{y}=\boldsymbol{u}_k$, $\boldsymbol{z}=A^\mathsf{T}\boldsymbol{y}$ to obtain $\boldsymbol{z}_k=B^{(\ell)}\boldsymbol{u}_k$. \\
		\STATE $\boldsymbol{x}_k=\boldsymbol{x}_0+\boldsymbol{z}_k$
	\end{algorithmic}
\end{algorithm}

In the AB-GMRES preconditioned by inner iterations, one common choice for stationary inner iterations is the NE-SOR method, which is mathematically equivalent to the SOR method applied to the normal equations of the second kind~\cite{BE79,S03}.
More specifically, if we use ${\boldsymbol{\alpha} _i^\mathsf{T}}$ to represent the $i$th row of the matrix $A$, and $v_i$ to represent the $i$th entry of the vector $\boldsymbol{v}$, then the NE-SOR method for $AA^\mathsf{T}\boldsymbol{y}=\boldsymbol{v}, \boldsymbol{z}=A^\mathsf{T}\boldsymbol{y}$ can be described as follows.

\begin{algorithm}[H]
	\caption{NE-SOR method~\cite{S03}}
	\label{NE-SOR}
	\begin{algorithmic}[1]
		\STATE Let $\boldsymbol{z}^{(0)}$ be the initial approximate solution and $\omega \in {\mathbb{R}}$ be the relaxation parameter.  \\
		\FOR{$p = 0,1,2, \ldots, \ell-1$} 
		\FOR{$i = 1,2, \ldots, m$}
		\STATE ${\boldsymbol{z}^{(p)}} = {\boldsymbol{z}^{(p)}} + \omega \frac{{{v_i} - \boldsymbol{\alpha} _i^\mathsf{T}{\boldsymbol{z}^{(p)}}}}{{\left\| {{\boldsymbol{\alpha} _i}} \right\|_2^2}}{\boldsymbol{\alpha} _i}$
		\ENDFOR
		\STATE ${\boldsymbol{z}^{(p + 1)}} = {\boldsymbol{z}^{(p)}}$		
		\ENDFOR
	\end{algorithmic}
\end{algorithm}

The Kaczmarz method~\cite{K37} is equivalent to Algorithm~\ref{NE-SOR} with $\omega=1$.
In fact, the iteration scheme of NE-GS (NE-SOR) is exactly the same as that of the Kaczmarz method (relaxed Kaczmarz method)~\cite{K37}. 
The relaxed Kaczmarz (NE-SOR) method is one of the most efficient row action methods. For $\omega=1$, the method cycles through the rows of the linear system and forms each iterate by orthogonally projecting the current point onto the hyperplane $\boldsymbol{\alpha} _{i}^\mathsf{T}\boldsymbol{z}^{(p+1)}=v_{i}$ formed by the active row, and all the $m$ equations in the linear system are swept through consecutively in $m$ iterations.

The convergence theorem of AB-GMRES preconditioned by NE-SOR
is precisely restated below.

\begin{theorem}~\cite[Theorem~5.6]{MH15}.
	AB-GMRES preconditioned by NE-SOR inner iterations with $0<\omega<2$, determines the minimum-norm solution of $A\boldsymbol{x}=\boldsymbol{b}$ without breakdown for all $\boldsymbol{b} \in \mathcal{R}(A)$ and for all ${\boldsymbol{x}_0} \in \mathcal{R}(A^\mathsf{T})$.
\end{theorem}

The condition ${\boldsymbol{x}_0} \in {\mathbb{R}^{n}}$ in ~\cite[Theorem{\color{black}s}~5.5 and 5.6]{MH15} should be ${\boldsymbol{x}_0} \in {\mathcal{R}({A^\mathsf{T}})}.$ This follows from ~\cite[Theorem~5.2]{MH15}.


\subsection{Flexible AB-GMRES method}
We adopt the flexible preconditioning framework proposed in FGMRES~\cite{S93} to AB-GMRES, and consider using Kaczmarz-type inner iterations in it.

\subsubsection{Outer-iteration algorithm}
It is well known that the preconditioning matrix needs to be fixed in the preconditioned GMRES method for all the outer iterations.
In fact, in order to keep the preconditioner ${B^{(\ell)}}$ in Algorithm~\ref{ABPI} fixed, the number of inner iterations in AB-GMRES should not be changed for each outer iteration.
However, if we were to adopt the randomized or greedy algorithm as the inner-iteration preconditioner in AB-GMRES, the preconditioning matrix for each outer iteration may change even though the number of inner iterations for each outer iteration is fixed.
In~\cite{S93}, Saad presented a variant of the GMRES algorithm called flexible GMRES (FGMRES) for solving square linear systems $A\boldsymbol{x}=\boldsymbol{b}$, which allows the preconditioning matrix to change for each outer iteration. A similar variant of AB-GMRES can be described as follows, where $B^{(\ell_k)}$ denotes the preconditioning matrix for $\ell_k$ inner iterations for the $k$th outer iteration.
\begin{algorithm}[H]
	\caption{Flexible AB-GMRES (F-AB-GMRES) method}
	\label{Flexible GMRES(FGMRES)}
	\begin{algorithmic}[1]
		\STATE Let $\boldsymbol{x}_0$ be the initial approximate solution and $\boldsymbol{r}_0=\boldsymbol{b}-A\boldsymbol{x}_0$.
		\STATE $\beta  = {\left\| {{\boldsymbol{r}_0}} \right\|_2}$, ${\boldsymbol{v}_1} = {\boldsymbol{r}_0}/\beta$
		\FOR{$k = 1,2, \ldots$ until convergence} 
		\STATE $\boldsymbol{z}_k=B^{({\ell_k})}{\boldsymbol{v}_k}$, ${\boldsymbol{w}_k} = A\boldsymbol{z}_k$
		\FOR{$i = 1,2, \ldots ,k,$}
		\STATE ${h_{i,k}} = ({\boldsymbol{w}_k },{\boldsymbol{v}_i})$, ${\boldsymbol{w}_k } = {\boldsymbol{w}_k } - {h_{i,k}}{\boldsymbol{v}_i}$
		\ENDFOR
		\STATE ${h_{k + 1,k}} = {\left\| {{\boldsymbol{w}_k }} \right\|_2}$, ${\boldsymbol{v}_{k + 1}} = {\boldsymbol{w}_k }/{h_{k + 1,k}}$
		\STATE Define ${Z_k} = \left[ {{\boldsymbol{z}_1}, {\boldsymbol{z}_2}, \ldots ,{\boldsymbol{z}_k}} \right]$.
		\ENDFOR
		\STATE ${\boldsymbol{y}_k} = \arg \mathop {\min }_{\boldsymbol{y} \in {\mathbb{R}^k}} {\left\| {\beta {\boldsymbol{e}_1} - {{\bar H}_k}\boldsymbol{y}} \right\|_2}$, where ${{\bar H}_k} = {\left\{ {{h_{ij}}} \right\}_{1 \leq i \leq k + 1;1 \leq j \leq k}}$
		\STATE ${\boldsymbol{x}_k} = {\boldsymbol{x}_0} + Z_k{\boldsymbol{y}_k}$
	\end{algorithmic}
\end{algorithm}
Here, $Z_k$ denotes the $n \times k$ matrix with column vectors $\boldsymbol{z}_1$, $\boldsymbol{z}_2$, $\ldots$, $\boldsymbol{z}_k$. For later use, let $H_k$ denote the $k \times k$ matrix obtained from ${{\bar H}_k}$ by deleting its last row, and $V_k$ denote the $m \times k$ matrix with column vectors $\boldsymbol{v}_1$, $\boldsymbol{v}_2$, $\ldots$, $\boldsymbol{v}_k$.

In the following, we will propose using Kaczmarz-type algorithms as $B^{(\ell_k)}$ in Algorithm~\ref{Flexible GMRES(FGMRES)}.

\subsubsection{Kaczmarz-type inner-iteration algorithms}
The (relaxed) Kaczmarz method is given below.
We will call it the Kaczmarz method for short in the following.
\begin{algorithm}[H] 
	\caption{Kaczmarz (K) method~\cite{K37}} \label{KMM}
	\begin{algorithmic}[1]
		\STATE Let ${\boldsymbol{z}^{(0)}}$ be the initial approximate solution. 
		\FOR{$p = 0,1,2, \ldots, \ell-1$}
		\STATE ${i_p} = (p\bmod m)+1$
		\STATE ${\boldsymbol{z}^{(p + 1)}} = {\boldsymbol{z}^{(p)}} + \omega \frac{{{v_{{i_p}}} - \boldsymbol{\alpha} _{{i_p}}^\mathsf{T}{\boldsymbol{z}^{(p)}}}}{{\left\| {{\boldsymbol{\alpha} _{{i_p}}}} \right\|_2^2}}{\boldsymbol{\alpha} _{{i_p}}}$	
		\ENDFOR
	\end{algorithmic}
\end{algorithm}

Instead of using the rows of the coefficient matrix $A$ consecutively, Ansorge~\cite{A84} proposed selecting the $i_p$th row corresponding to the residual component with maximum absolute value for the $p$th iteration. We will call this method the greedy Kaczmarz method, or GK method for short.
See also~\cite{S40} for earlier work.
Another approach is to choose the $i_p$th row randomly for the $p$th iteration~\cite{SV09}, which is called the randomized Kaczmarz (RK) method.
Finally, Bai and Wu~\cite{BW18} combined these two ideas to propose a more effective probability criterion, which aims to
diminish entries of the residual vector with relatively large absolute value at each iteration. We will call the corresponding algorithm the greedy randomized Kaczmarz (GRK) method. These algorithms are called Kaczmarz-type methods. It can be proved that GK, RK and GRK methods converge to the minimum-norm solution whether the system is overdetermined or underdetermined and the coefficient matrix has full rank or is rank deficient~\cite{P18}. The algorithms of GK, RK and GRK are given below.


\begin{algorithm}[H]
	\caption{Greedy Kaczmarz (GK) method~\cite{A84}}
	\label{The Greedy Kaczmarz Method}
	\begin{algorithmic}[1]
		\STATE Let ${\boldsymbol{z}^{(0)}}$ be the initial approximate solution. 
		\FOR{$p = 0,1,2, \ldots, \ell-1$}
		\STATE Select ${i_p} \in \left\{ {1,2, \ldots m} \right\}$ according to ${i_p} = \arg \mathop {\max }\limits_i  {{{| {v_i} - \boldsymbol{\alpha} _i^\mathsf{T}{\boldsymbol{z}^{(p)}} |}}}$.		
		\STATE ${\boldsymbol{z}^{(p + 1)}} = {\boldsymbol{z}^{(p)}} + \omega \frac{{{v_{{i_p}}} - \boldsymbol{\alpha} _{{i_p}}^\mathsf{T}{\boldsymbol{z}^{(p)}}}}{{\left\| {{\boldsymbol{\alpha} _{{i_p}}}} \right\|_2^2}}{\boldsymbol{\alpha} _{{i_p}}}$	
		\ENDFOR
	\end{algorithmic}
\end{algorithm}

\begin{algorithm}[H] 
	\caption{Randomized Kaczmarz (RK) method~\cite{SV09}} \label{RKMM}
	\begin{algorithmic}[1]
		\STATE Let ${\boldsymbol{z}^{(0)}}$ be the initial approximate solution. 
		\FOR{$p = 0,1,2, \ldots, \ell-1$}
		\STATE Select ${i_p} \in \left\{ {1,2, \ldots m} \right\}$ with probability $\Pr (\text{row} = {i_p}) = \frac{{\left\| {{\boldsymbol{\alpha} _{{i_p}}}} \right\|_2^{2}}}{{\left\| A \right\|_\text{F}^{2}}}$.	
		\STATE ${\boldsymbol{z}^{(p + 1)}} = {\boldsymbol{z}^{(p)}} + \omega \frac{{{v_{{i_p}}} - \boldsymbol{\alpha} _{{i_p}}^\mathsf{T}{\boldsymbol{z}^{(p)}}}}{{\left\| {{\boldsymbol{\alpha} _{{i_p}}}} \right\|_2^2}}{\boldsymbol{\alpha} _{{i_p}}}$	
		\ENDFOR
	\end{algorithmic}
\end{algorithm}

\begin{algorithm}[H] 
	\caption{Greedy Randomized Kaczmarz (GRK) method~\cite{BW18}} \label{GRKMM}
	\begin{algorithmic}[1]
		\STATE Let ${\boldsymbol{z}^{(0)}}$ be the initial approximate solution. 
		\FOR{$p = 0,1,2, \ldots, \ell-1$}
		\STATE ${\varepsilon _p} = \frac{1}{2}\left( {\frac{1}{{\left\| {\boldsymbol{v} - A{\boldsymbol{z}^{(p)}}} \right\|_2^2}}\mathop {\max }\limits_{1 \le {i_p} \le m} \left\{ {\frac{{{{\left| {{v_{{i_p}}} - \boldsymbol{\alpha} _{{i_p}}^\mathsf{T}{\boldsymbol{z}^{(p)}}} \right|}^2}}}{{\left\| {{\boldsymbol{\alpha} _{{i_p}}}} \right\|_2^{\rm{2}}}}} \right\} + \frac{1}{{\left\| A \right\|_\mathsf{F}^{\rm{2}}}}} \right)$
		\STATE Determine the index set of positive integers\\
		$${U_p} = \left\{ { {{i_p}} : {{{| {{v_{{i_p}}} - \boldsymbol{\alpha} _{{i_p}}^\mathsf{T}{\boldsymbol{z}^{(p)}}} |}^2}} \ge {\varepsilon _p}\| {\boldsymbol{v} - A{\boldsymbol{z}^{(p)}}} \|_2^2}{{\| {{\boldsymbol{\alpha} _{{i_p}}}} \|_2^{2}}} \right\}.$$
		\STATE Compute the $i$th entry $\tilde s_i^{(p)}$ of the vector ${{\boldsymbol{ \tilde s}}^{(p)}}$ according to\\
		\begin{equation*}
		\tilde s_i^{(p)}=
		\begin{cases}
		{v_i} - \boldsymbol{\alpha} _i^\mathsf{T}{\boldsymbol{z}^{(p)}},& \text{if} \quad i \in {U_p},\\
		0,& \text{otherwise.}
		\end{cases}
		\end{equation*}
		\STATE Select ${i_p} \in {U_p}$ with probability $\Pr (\text{row} = {i_p}) = \frac{{{{\left| {\tilde s_{i_p}^{(p)}} \right|}^2}}}{{\left\| {{{ \boldsymbol{\tilde s}}^{(p)}}} \right\|_2^{\text{2}}}}$.
		\STATE ${\boldsymbol{z}^{(p + 1)}} = {\boldsymbol{z}^{(p)}} + \omega \frac{{{v_{{i_p}}} - \boldsymbol{\alpha} _{{i_p}}^\mathsf{T}{\boldsymbol{z}^{(p)}}}}{{\left\| {{\boldsymbol{\alpha} _{{i_p}}}} \right\|_2^2}}{\boldsymbol{\alpha} _{{i_p}}}$	
		\ENDFOR	
	\end{algorithmic}
\end{algorithm}

In Algorithms~\ref{RKMM} and~\ref{GRKMM}, Pr($\text{row}=i$) represents the probability of selecting the $i$th row of the matrix $A$ as the working row of this iteration.

We remark that $\boldsymbol{s}=\boldsymbol{v}-A\boldsymbol{z}$ needs to be calculated at each inner iteration for GK and GRK methods. This additional computational work cannot be ignored. On the other hand, we may update the residual vector $\boldsymbol{s}$ recursively as follows~\cite{BW18}:
\begin{align}\label{r}
{\boldsymbol{s}^{(p+1)}} &= \boldsymbol{v} - A{\boldsymbol{z}^{(p+1)}}\nonumber\\
&=\boldsymbol{v} - A\left( {{\boldsymbol{z}^{(p)}} + {\color{black}{\omega}}\frac{{{v_{{i_p}}} - \boldsymbol{\alpha} _{{i_p}}^\mathsf{T}{\boldsymbol{z}^{(p)}}}}{{\left\| {{\boldsymbol{\alpha} _{{i_p}}}} \right\|_2^2}}\boldsymbol{\alpha} _{{i_p}}} \right)\nonumber\\
&= \boldsymbol{v} - A{\boldsymbol{z}^{(p)}} - {\color{black}{\omega}}\frac{{s_{{i_p}}^{\left( p \right)}}}{{\left\| {{\boldsymbol{\alpha} _{{i_p}}}} \right\|_2^{\rm{2}}}}A\boldsymbol{\alpha} _{{i_p}}\nonumber\\
&={\boldsymbol{s}^{\left( p \right)}}-{\color{black}{\omega}}\frac{{s_{{i_p}}^{\left( p \right)}}}{{\left\| {{\boldsymbol{\alpha} _{{i_p}}}} \right\|_2^{\rm{2}}}}{C_{({i_p})}}.
\end{align}
Here, ${C_{({i_p})}}$ is the ${i_p}$th column of $C=AA^\mathsf{T}$.
Hence, if the matrix product $AA^\mathsf{T}$ is computed and stored once in the beginning, the computational work can be reduced, assuming that the total number of inner iterations is more than the number of rows of $A$. (See Appendix.) This condition was satisfied in all our numerical experiments.


\section{Flexible AB-GMRES preconditioned by Kaczmarz-type methods as inner iterations}
In FGMRES,  we can change the preconditioner for each outer iteration.
Hence, consider using $\ell_{k}$ iterations of a Kaczmarz-type method as the preconditioner for each outer iteration of the flexible AB-GMRES (F-AB-GMRES). We denote the preconditioning matrix given by the $\ell_{k}$ inner iterations by ${B^{(\ell_{k})}}$. The algorithm is given as follows.
\begin{algorithm}[H]
	\caption{Flexible AB-GMRES preconditioned by Kaczmarz-type methods}
	\label{AGPK}
	\begin{algorithmic}[1]
		\STATE Let $\boldsymbol{x}_0$ be the initial approximate solution and $\boldsymbol{r}_0=\boldsymbol{b}-A\boldsymbol{x}_0$.
		\STATE $\beta  = {\left\| {{{\boldsymbol{r}}_0}} \right\|_2}$, ${\boldsymbol{v}_1} = {{\boldsymbol{r}}_0}/\beta$
		\FOR{$k = 1,2, \ldots$ until convergence} 
		\STATE Apply $\ell_{k}$ iterations of a Kaczmarz-type method to $A\boldsymbol{z}\!=\!\boldsymbol{v}_k$ to obtain $\boldsymbol{z}_k\!=\!{B^{(\ell_{k})}}\boldsymbol{v}_k$, where $\ell_{k}$$=\min \{\ell_\text{max}, \ell\}$, and $\ell_\text{max}$ is the maximum number of inner iterations allowed, and $\ell$ is the smallest $\ell$ such that\\
		$${\left\| {\boldsymbol{v}_k - A{B^{(\ell)}}\boldsymbol{v}_k} \right\|_2} \le {\color{black}{\eta}} {\left\| \boldsymbol{v}_k \right\|_2}.$$
		\STATE ${\boldsymbol{w}_k} = A{\boldsymbol{z}_k}$ 
		\FOR{$i = 1,2, \ldots ,k,$}
		\STATE ${h_{i,k}} = {\boldsymbol{w}_k}^\mathsf{T}{\boldsymbol{v}_i}$, ${\boldsymbol{w}_k} = {\boldsymbol{w}_k} - {h_{i,k}}{\boldsymbol{v}_i}$
		\ENDFOR
		\STATE ${h_{k + 1,k}} = {\left\| {{\boldsymbol{w}_k}} \right\|_2}$, ${\boldsymbol{v}_{k + 1}} = {\boldsymbol{w}_k}/{h_{k + 1,k}}$
		\ENDFOR
		\STATE ${\boldsymbol{y}_k} \!=\! \arg \mathop {\min }_{\boldsymbol{y} \in {\mathbb{R}^k}} {\left\| {\beta {\boldsymbol{e}_1} - {{\bar H}_k}\boldsymbol{y}} \right\|_2}$, ${\boldsymbol{u}_k} \!=\! \left[ {{\boldsymbol{z}_1},{\boldsymbol{z}_2}, \ldots ,{\boldsymbol{z}_k}} \right]{\boldsymbol{y}_k}$, where ${{\bar H}_k} \!=\! {\left\{ {{h_{ij}}} \right\}_{1 \leq i \leq k + 1;1 \leq j \leq k}}$
		\STATE $\boldsymbol{x}_k=\boldsymbol{x}_0+\boldsymbol{u}_k$
	\end{algorithmic}
\end{algorithm}
Here, ${{\bar H}_k} = \left\{ {{h_{i,j}}} \right\} \in {\mathbb{R}^{(k + 1) \times k}}$. Since the number of inner iterations in each outer iteration does not have to be fixed, we proposed a new inner iterations stopping criterion that varies with the outer iteration to accelerate the convergence, which is given in line 4 of Algorithm~\ref{AGPK}.
Here, $\eta<1$ is a parameter. 

Note that when the Kaczmarz method is used as inner iterations, the number of inner iterations $\ell_k$ for each outer iteration $k$ does not have to be fixed to $\ell_\text{max}$, and may differ for each outer iteration, which makes it different from NE-SOR inner iterations applied to AB-GMRES.
Note also that \eqref{r} may be used to speed up the residual evaluation in step 4 of Algorithm~3.1 for the Kaczmarz-type methods.

The least squares problem in line 11 is solved as in the GMRES method~\cite{S86,S03}.
An optimality property similar to GMRES is given as in~\cite[Proposition~2.1]{S93} for FGMRES.

\begin{theorem}\label{the1}
	The approximate solution $\boldsymbol{x}_k$ obtained at the $k$th iteration of F-AB-GMRES minimizes the residual norm ${\left\| {\boldsymbol{b} - A{\boldsymbol{x}_k}} \right\|_2}$ over ${\boldsymbol{x}_0} + \mathrm{span}\left\{ {\boldsymbol{z}_1},{\boldsymbol{z}_2}, \ldots ,{\boldsymbol{z}_k} \right\}$.
\end{theorem}

\begin{theorem}
	If ${\boldsymbol{z}^{\left( 0 \right)}},{\boldsymbol{x}_0} \in \mathcal{R}({A^\mathsf{T}})$, when F-AB-GMRES preconditioned by one of the above Kaczmarz-type methods gives a solution of $A\boldsymbol{x}=\boldsymbol{b}$, it is the minimum Euclidean-norm solution.
\end{theorem}

\begin{proof}
	In K, GK, RK and GRK, if the initial iterate ${\boldsymbol{z}^{\left( 0 \right)}} \in \mathcal{R}({A^\mathsf{T}})$, then ${\boldsymbol{z}^{\left( p \right)}} \in \mathcal{R}({A^\mathsf{T}})$. Hence, if ${\boldsymbol{x}_0} \in \mathcal{R}({A^\mathsf{T}})$, the F-AB-GMRES iterate ${\boldsymbol{x}_k} \in \mathcal{R}({A^\mathsf{T}})$, since ${{\color{black}{\boldsymbol{u}}}_k} \in \text{span}\{ {\boldsymbol{z}_1}, {\boldsymbol{z}_2}, \ldots ,{\boldsymbol{z}_k}\}$. Therefore, when $\boldsymbol{x}_k$ is a solution of $A\boldsymbol{x}=\boldsymbol{b}$, it is the minimum Euclidean-norm solution since ${\boldsymbol{x}_k} \in \mathcal{R}({A^\mathsf{T}}) = \mathcal{N}{(A)^ \bot }$.
\end{proof}

Next, we consider the possibility of breakdown in F-AB-GMRES preconditioned by Kaczmarz-type methods as inner iterations.
A breakdown occurs when the vector $\boldsymbol{v}_{k+1}$ cannot be computed in line 9 of Algorithm~\ref{AGPK} because $h_{k+1, k} = 0$. For AB-GMRES with $B$ satisfying the convergence conditions~\cite[Theorem~5.2]{MH15},~\cite[Corollary~3.8]{HYI10}, this is not a problem because when this happens, the approximate solution $\boldsymbol{x}_k$ satisfies $A\boldsymbol{x}_k=\boldsymbol{b}$. 
The situation for F-AB-GMRES is slightly different, as in~\cite[Proposition~2.2]{S93} for FGMRES.

\begin{theorem}\label{the2}
	Assume that $\beta=\left\|\boldsymbol{r}_{0}\right\|_{2} \neq 0$ and that $k-1$ steps of F-AB-GMRES have been successfully performed, i.e., that $h_{i+1, i} \neq 0$ for $i < k$. In addition, assume that the matrix $H_k$ is nonsingular. Then $\boldsymbol{x}_k$ is a solution of $A\boldsymbol{x}=\boldsymbol{b}$ if and only if $h_{k+1, k} = 0$.
\end{theorem}
%
%

The only difference of this result from that of AB-GMRES is that the additional assumption that $H_k$ is nonsingular must be made since it is no longer implied by the algorithm.
In fact, the following holds

\begin{theorem}
	Assume $\rank A = n$, $\rank Z_k =k$, and $h_{k+1,k} = 0.$ Then, $H_k$ is nonsingular.
\end{theorem}

\begin{proof}
	Let $c_1 A \boldsymbol{z}_1 + \cdots + c_k A \boldsymbol{z}_k = A(c_1 \boldsymbol{z}_1 + \cdots + c_k \boldsymbol{z}_k ) = 0$. Then, $ \rank A = n$ implies $c_1 \boldsymbol{z}_1 + \cdots + c_k \boldsymbol{z}_k = 0$, and $ \rank Z_k = k $ implies $c_1 = \cdots =c_k = 0$. Hence, $ \rank(AZ_k)=k$. Since, $h_{k+1,k} = 0$, $AZ_k = V_k H_k$. Thus, $k=\rank(AZ_k)=\rank(V_k H_k) \le \min(\rank V_k, \rank H_k ) = \min (k, \rank H_k)$. Hence, $\rank H_k = k$, and $H_k$ is nonsingular. 
\end{proof}

See also~\cite[Theorem~3]{M17} for convergence conditions of FGMRES preconditioned by multistep matrix splitting iterations, which ensure the nonsingularity of $H_k$.

The additional cost of the flexible variant over AB-GMRES is only the extra memory required to save the set of vectors $\{\boldsymbol{z}_j\}_{j=1,2,\ldots,m}$. 
However, the added advantage of flexibility may be worth this extra cost.

The computational work of one iteration for NE-SOR, K, RK, GRK and GK are summarized in Table~\ref{tab:0}, respectively. Note that we assume that ${{\left\| {{\boldsymbol{\alpha} _i}} \right\|_2^2}}$ is precomputed and stored. Here, $q=\text{nz}/m$ where $\text{nz}$ is the number of nonzero elements of $A$ (i.e. $q$ is the average number of nonzero elements per row of $A$) and $p$ is the density of nonzero elements of $C=AA^\mathsf{T}$ (See Appendix A for the estimation of $p$). Also, note that the Kaczmarz-type methods require computing $C=AA^\mathsf{T}$ once beforehand, which amounts to $qm^2$ floating point operations. Note also that we regard a pair of multiplication and addition as one floating-point operation.

\begin{table}[!t]
	\caption{Number of floating point operations required for one inner iteration.}\label{tab:0}
	\begin{center}
		\begin{tabular}{rr}
			\hline
			\multicolumn{1}{c}{Method} & \multicolumn{1}{c}{No. of floating point operations} \\
			\hline
			NE-SOR & $2q$  \\
			K &  $m(1+p)+q$\\
			GK & $m(1+p)+q$ \\
			RK & $m(1+p)+q$ \\
			GRK & $m(3+p)+q$ \\
			\hline
		\end{tabular}
		\begin{tablenotes}
			$q=\text{nz}/m$, where $\text{nz}$ is the number of nonzero elements of the matrix $A$ and $m$ is the number of rows of the matrix $A$, $p$ is the density of nonzero elements of the matrix $C=AA^\mathsf{T}$.
		\end{tablenotes}
	\end{center}
\end{table}

The number of floating point operations for the $k$th outer iteration (other than the inner iterations) of AB-GMRES and F-AB-GMRES is approximately $m(q+2k+2)$.


\section{Numerical experiments}
We compare the proposed F-AB-GMRES preconditioned by Kaczmarz-type methods as inner iterations with AB-GMRES preconditioned by NE-SOR inner iterations~\cite{MH15} in terms of the central processing unit (CPU) time by numerical experiments on underdetermined and overdetermined problems.

Table~\ref{tab:1} gives the number of rows $m$, the number of columns $n$, the density of the nonzero elements, the rank and the condition number $\kappa(A)$ on the overdetermined test matrices.
The matrices RANDL$i$, $i=1,2,\ldots,6$ were randomly generated using the MATLAB function \texttt{sprandn}, as in~\cite{HYI10,MH13,MH15}. 
The illc1850, gen and photogrammetry2 are full-rank matrices from~\cite{DH11}.
The Maragal$\_j$, $j= 3,4,5$ are rank-deficient matrices from~\cite{DH11}.
These matrices were transposed to form underdetermined problems. Table~\ref{tab:1} shows the effective size of the matrices after removing all zero rows. (If the matrix $A$ has a zero row, then the Kaczmarz-type methods can not work.) 
The condition number was computed using the MATLAB function \texttt{svd}.

In our implementations, a solution vector $\boldsymbol{x}_{\star} \in {\mathbb{R}^{n}} $ is randomly generated by using the MATLAB function \texttt{randn}, and the right-hand side $\boldsymbol{b} \in {\mathbb{R}^{m}}$ is taken to be $A\boldsymbol{x}_{\star}$.
All computations are started from the initial vector $\boldsymbol{x}_0=0$, and the iterations are stopped when either the relative residual
\begin{equation*}
\frac{{{{\left\| {{{\boldsymbol{r}}_k}} \right\|}_2}}}{{{{\left\| \boldsymbol{b} \right\|}_2}}} \le {10^{ - 6}},
\end{equation*}
where ${\left\| {{\boldsymbol{r}_k}} \right\|_2} = {\left\| {\boldsymbol{b} - A{\boldsymbol{x}_k}} \right\|_2} = {\left\| {\beta {\boldsymbol{e}_1} - {{\bar H}_k}{\boldsymbol{y}_k}} \right\|_2}$ in Algorithm \ref{AGPK}, or the number of outer iterations reaches 2000. 
The latter is given a label `$--$' in the tables showing the numerical experiment results.
No restarts were used for GMRES.
For the Kaczmarz-type inner iterations, {\color{black}the initial vector} $\boldsymbol{z}^{(0)}=0$ was used. 

All experiments were carried out using MATLAB (version R2020b) on a personal computer with 1.80 GHz CPU (Intel(R) Core(TM) i5-8265U), 32 GB memory, and Microsoft Windows 10 Pro 64 bit Version 1909 with DirectX 12. 


\subsection{Automatic parameter tuning for Kaczmarz-type methods}
The proposed method requires two preconditioning parameters: the maximum number of inner iterations $\ell_\text{max}$ in line 4 of Algorithm~\ref{AGPK} and the relaxation parameter $\omega$ used for the Kaczmarz-type inner iterations. Since the CPU time for the proposed method varies with the values of these parameters, it is desirable to determine the values automatically for any problem. Inspired by the idea in~\cite{MH15}, we perform the following procedure given as Algorithm~\ref{PTPK} using Kaczmarz-type methods alone for $A\boldsymbol{z}=\boldsymbol{b}$ before starting the outer iterations to determine the values of these parameters $\ell_\text{max}$ and $\omega_\text{opt}$.
Note that $\ell$ NE-SOR inner iterations of Algorithm~\ref{NE-SOR} is equivalent to $\ell m$ Kaczmarz inner iterations.


\begin{algorithm}[H]
	\caption{Parameter tuning procedure of Kaczmarz-type methods}
	\label{PTPK}
	\begin{algorithmic}[1]
		\STATE Set $\omega=1$ and $\boldsymbol{z}^{(0)}=0$.
		\STATE Apply a Kaczmarz-type method to $A\boldsymbol{z}=\boldsymbol{b}$
		{\color{black}until the $\ell$th iteration $\boldsymbol{z}^{(\ell)}$ satisfies}
		$${\left\| {\boldsymbol{b} - A{\boldsymbol{z}^{(\ell)}}} \right\|_2} \leq \eta {\left\| \boldsymbol{b} \right\|_2}.$$
		\STATE $\ell_\text{max} = \ell$
		\FOR{$\omega = 0.1, 0.2, \ldots, 1.9$} 
		\STATE Apply $\ell_\text{max}$ iterations of a Kaczmarz-type method to $A\boldsymbol{z}=\boldsymbol{b}$.
		\ENDFOR
		\STATE ${\omega _\text{opt}} = \arg \mathop {\min }_{\omega=0.1,0.2,\ldots,1.9} {\left\|
			{\boldsymbol{b} - A{\boldsymbol{z}^{(\ell_\text{max})}}} \right\|_2}$
	\end{algorithmic}
\end{algorithm}
Since the number of inner iterations in F-AB-GMRES does not have to be fixed for each outer iteration, we set $\ell_\text{max}$ determined by Algorithm~\ref{PTPK} to be the maximum number of inner iterations for each outer iteration.
Since the parameter $\ell_\text{max}$ of RK and GRK changed from one time to another, we repeated steps 2--7 of Algorithm~\ref{PTPK} ten times and took the median. For AB-GMRES, we set $\ell_\text{max}$ determined by Algorithm~\ref{PTPK} to be the fixed number of inner iterations for each outer iteration.

We tested the above procedure for the martrix Maragal$\_$3T as shown in Table~\ref{tab:eta}. Different values for $\eta$ in the procedure were used: $\eta=10^{-\mu}$, $\mu=2, 1, 0.5$.
Table~\ref{tab:eta} gives the numerical experiment results with different $\eta$ for matrix Maragal$\_$3T.
The first row in each cell in Table~\ref{tab:eta} gives the number of outer iterations outside parentheses and gives the total number of inner iterations (the sum of the number{\color{black}s} of inner iterations in each outer iteration)
and relaxation parameter in parentheses.
Here, the number of inner iterations is $\ell m$ for Algorithm~\ref{NE-SOR}, and $\ell_{k}$ for line 4 of Algorithm~\ref{AGPK}.
The second row in each cell gives the total CPU time in seconds including the parameter tuning time and the formation of $C$ outside parentheses, and the parameter tuning time in seconds in parentheses. Here, {\color{black}the number of iterations and} the CPU time mean the median of {\color{black}{the required iterations and}} the elapsed CPU time for ten times of repeated runs for the same $\boldsymbol{b}$ for the corresponding method, respectively. The $\star$ indicates the fastest method regarding the CPU time.
The third row in each cell gives the relative error norm ${\left\| {{\boldsymbol{x}_k} - {\boldsymbol{x}_ *  }} \right\|_2} / {\left\| {{\boldsymbol{x}_ * }} \right\|_2}$, where $\boldsymbol{x}_ *=A^\dag \boldsymbol{b}$  is the minimum Euclidean-norm solution. In our implementations, $A^\dag$ is obtained by the MATLAB function \texttt{pinv}. 

We remark that Morikuni and Hayami~\cite{MH15} evaluated the performance of AB-GMRES preconditioned by NE-SOR inner iterations for 
different values of $\eta$ and finally chose $\eta=10^{-1}$. 
In the following, {\color{black}also} $\eta=10^{-1}$ was used in Algorithm~\ref{AGPK} {\color{black}and~\ref{PTPK}} with the mentioned procedure to automatically tune the values of the parameters $\ell_\text{max}$ and $\omega_\text{opt}$.

\begin{table}[!t] 
	\caption{Information of the matrices.}\label{tab:1}
	\begin{center}
		\begin{tabular}{rrrrrr}
			\hline
			\multicolumn{1}{c}{Name} & \multicolumn{1}{c}{$m$} & \multicolumn{1}{c}{$n$} &\multicolumn{1}{c}{ Density[\%]} & \multicolumn{1}{c}{Rank} & \multicolumn{1}{c}{$\kappa(A)$} \\
			\hline
			RANDL1 & 5000 & 500 & 20 & 500 & 1.00$\times 10^{1}$ \\
			RANDL2 & 5000 & 500 & 20 & 500 & 1.00$\times 10^{2}$ \\
			RANDL3 & 5000 & 500 & 20 & 500 & 1.00$\times 10^{3}$ \\
			RANDL4 & 5000 & 500 & 20 & 500 & 1.00$\times 10^{4}$ \\
			RANDL5 & 5000 & 500 & 20 & 500 & 1.00$\times 10^{5}$ \\
			RANDL6 & 5000 & 500 & 20 & 500 & 1.00$\times 10^{6}$  \\
			illc1850 & 1850 & 712 & 0.66 & 712 & 1.40$\times 10^{3}$ \\
			gen & 2561 & 769 & 3.20 & 769 & 27.72 \\
			photogrammetry2 & 4472 & 936 & 0.89 & 936 & 1.34$\times 10^{8}$ \\
			Maragal$\_$3 & 1682 & 858 & 1.27 & 613 & 1.10$\times 10^{3}$ \\
			Maragal$\_$4 & 1964 & 1027 & 1.32 & 801 & 9.33$\times 10^{6}$ \\
			Maragal$\_$5 & 4654& 3296 & 0.61 & 2147 & 1.19$\times 10^{5}$ \\
			\hline
		\end{tabular}
		\begin{tablenotes}
			Name: name of the matrix, $m$: number of rows of the matrix, $n$: number of columns of the matrix, Density: density of the nonzero components of the matrix, Rank: maximum number of linearly independent columns of the matrix, obtained by the MATLAB command \texttt{rank(full($A$))}, which is based on the singular value decomposition of $A$, $\kappa(A)$: condition number of the matrix ${\sigma _{\max }}/{\sigma _{\min }}$, where ${\sigma _{\max }}$ and ${\sigma _{\min }}$ are the largest and smallest nonzero singular values of the matrix, respectively.
		\end{tablenotes}
	\end{center}
\end{table}

\begin{table}[!t] 
	{\footnotesize
		\setlength{\tabcolsep}{3.8pt}
		\caption{Results with different values for $\eta$ for {\color{black}{the}} matrix Maragal$\_$3T.}\label{tab:eta}
		\begin{center}
			\begin{tabular}{ll|rrrrrr}
				\hline
				Outer & Inner &  &  & & & & \\
				iteration & iteration & \multicolumn{2}{c}{$\eta=10^{-2}$} & 
				\multicolumn{2}{c}{$\eta=10^{-1}$} & \multicolumn{2}{c}{$\eta=10^{-0.5}$} \\
				\hline
				AB-GMRES & NE-SOR &  18 & (401544, 1.3) & 57 & (195624, 
				1.2) & 86 & (147576, 0.9) \\
				& & 8.55 & (4.52) & 2.62 & (0.66) & 1.86 & (0.35) \\
				& & &1.96$\times 10^{-5}$  & &3.02$\times 10^{-5}$  & & 4.65$\times 10^{-5}$  \\
				\hline
				F-AB-GMRES & K & 19 & (420100, 1.2) & 68 & (176392, 1.0) & 127 & (113279, 
				0.8) \\
				& & 9.19 & (4.72) & 2.39 & (0.53) & 1.46 & (0.19) \\
				& & &2.46$\times 10^{-5}$  & &4.88$\times 10^{-5}$  & & 3.52$\times 10^{-5}$  \\
				\cline{2-8}
				& RK & 79.5 & (5388296, 1.3) & 284.5 & (1402900, 1.1) & 558.5 & (711934, 1.0) \\
				& & 355.36 & (95.24) & 73.81 & (6.93) & 36.93 & (1.79) \\
				& & &4.60$\times 10^{-5}$  & &7.38$\times 10^{-6}$  & & 8.07$\times 10^{-6}$  \\
				\cline{2-8}
				& GRK & 17.5& (123470, 1.3) & 67.5 & (44395, 1.2) & 166 & (34860, 1.1) \\
				& & 18.72 & (11.81) & 3.61 & 
				(1.10) & 2.48 & (0.37) \\
				& & &1.31$\times 10^{-5}$  & &1.61$\times 10^{-5}$  & & 1.98$\times 10^{-5}$  \\
				\cline{2-8}
				& GK & 27 & (176344, 1.3) & 129 & (83850, 1.0) & 345 & (71415, 1.0) \\
				& & \color{red}$\star$3.79 & (1.67) & \color{red}$\star$1.28 & 
				(0.18) & \color{red}$\star$1.20 & (0.06) \\
				& & &4.56$\times 10^{-5}$  & &3.53$\times 10^{-5}$  & & 6.10$\times 10^{-5}$  \\
				\hline
			\end{tabular}
	\end{center}}
	{\footnotesize
		\begin{center}
			\begin{tablenotes}
				\item[1] First row: Number of outer iterations (total number of inner iterations, $\omega$).
				\item[2] Second row: Total CPU time, which includes parameter tuning time in parentheses, in seconds.
				\item[3] Third row: Relative error norm.
			\end{tablenotes}
	\end{center}}
\end{table}

\subsection{Underdetermined problems}
We first present numerical experiment results on underdetermined problems ($m < n$).
Tables~\ref{tab:2},~\ref{tab:fullrank_under} and~\ref{tab:3} give the numerical experiment results on artificial random matrices, full-rank matrices and rank-deficient matrices, respectively.
The letter T at the end of the name of a matrix denotes the transposition of the matrix.
In order to improve computing efficiency for the Kaczmarz-type methods, we used the recursive formula~\eqref{r} to update the residual vector $\boldsymbol{s}^{(p)}$ for each inner iteration.
To do so, we need to compute $C=AA^\mathsf{T}$ in advance.
The CPU time in seconds for computing matrix $C$ is given below the name of the matrix.
The total CPU time for F-AB-GMRES preconditioned by the Kaczmarz-type inner iterations includes the time for computing $C$. We also remark that the column-oriented access to the matrix $A^\mathsf{T}$ instead of the row-oriented access to $A$ was used throughout the programs for efficient data access with MATLAB. (The CPU time required to transpose $A$ is negligible.)

\begin{table}[!t] 
	{\footnotesize
		\setlength{\tabcolsep}{4.8pt}
		\caption{Results for full-rank artificial random matrices (underdetermined 
			problems).}\label{tab:2}
		\begin{center}
			\begin{tabular}{ll|rrrrrr}
				\hline
				Outer & Inner & \multicolumn{2}{c}{RANDL1T} & 
				\multicolumn{2}{c}{RANDL2T} & \multicolumn{2}{c}{RANDL3T} \\
				iteration & iteration & \multicolumn{2}{c}{0.18} & 
				\multicolumn{2}{c}{0.17} & \multicolumn{2}{c}{0.18} \\
				\hline
				AB-GMRES & NE-SOR &  11 & (16500, 1.1) & 38 & (114000, 1.0) & 107 & (267500, 1.0) \\
				& & 1.18 & (0.76) & 4.34 & (1.50) & 8.04 & (1.26) \\
				& & &2.30$\times 10^{-7}$  & &5.82$\times 10^{-6}$  & & 9.33$\times 10^{-6}$  \\
				\hline
				F-AB-GMRES & K & 13 & (15080, 1.0) & 39 & (106431, 1.0) & 120 & (248280, 1.0) \\
				& & 1.16 & (0.59) & 4.27 & (1.39) & 7.65 & (1.06) \\
				& & &3.61$\times 10^{-7}$  & &7.51$\times 10^{-6}$  & & 1.19$\times 10^{-5}$   \\
				\cline{2-8}
				& RK & 10 & (29210, 1.0) & 47 & (226680, 1.0) & 212.5 & (1079100, 1.0) \\
				& & 7.21 & (5.20) & 22.31 & (8.50) & 75.08 & (8.93) \\
				& & &9.68$\times 10^{-7}$  & &9.67$\times 10^{-6}$  & & 1.55$\times 10^{-5}$   \\
				\cline{2-8}
				& GRK & 9 & (4746, 1.1) & 36.5 & (31458, 1.1) & 102.5 & (80378, 1.1) \\
				& & 1.64 & (1.12) & 4.15 & (1.76) & 7.53 & (1.60) \\
				& & &5.95$\times 10^{-7}$ & &7.29$\times 10^{-6}$  & & 2.16$\times 10^{-5}$   \\
				\cline{2-8}
				& GK & 9 & (4733, 1.1) & 36 & (32677, 1.3) & 106 & (79341, 1.1) \\
				& & \color{red}$\star$0.61 & (0.29) & \color{red}$\star$1.56 & 
				(0.49) & \color{red}$\star$2.91 & (0.42) \\
				& & &9.11$\times 10^{-7}$  & &5.23$\times 10^{-6}$  & & 2.81$\times 10^{-5}$   \\
				\hline
			\end{tabular}
	\end{center}}
	\medskip
	{\footnotesize
		\setlength{\tabcolsep}{4.4pt}
		\begin{center}
			\begin{tabular}{ll|rrrrrr}
				\hline
				Outer & Inner & \multicolumn{2}{c}{RANDL4T} & 
				\multicolumn{2}{c}{RANDL5T} & \multicolumn{2}{c}{RANDL6T} \\
				iteration & iteration & \multicolumn{2}{c}{0.17} & 
				\multicolumn{2}{c}{0.17} & \multicolumn{2}{c}{0.17} \\
				\hline
				AB-GMRES & NE-SOR & 72 & (180000, 0.9)& 199 & (497500, 
				0.9) & 107 & (321000, 1.0) \\
				& & 5.80 & (1.24) & 13.82 & (1.26) & 9.63 & (1.52) \\
				& & &1.34$\times 10^{-4}$  & &1.10$\times 10^{-3}$  & & 5.39$\times 10^{-2}$ \\
				\hline
				F-AB-GMRES & K & 75 & (163725, 1.0) & 203 & (440713, 1.0) & 118 & 
				(304558, 1.0) \\
				& & 5.54 & (1.12) & 12.59 & (1.11) & 9.39 & (1.39) \\
				& & &1.61$\times 10^{-4}$  & &5.10$\times 10^{-3}$  & & 5.39$\times 10^{-2}$  \\
				\cline{2-8}
				& RK & 144.5 & (630640, 1.1) & 346.5 & (1527016, 1.0) & 166.5 & 
				(899443, 1.1) \\
				& & 46.56 & (7.83) & 101.93 & (7.91) & 66.04 & (10.15) \\
				& & &1.09$\times 10^{-4}$  & &6.90$\times 10^{-3}$  & & 5.39$\times 10^{-2}$  \\
				\cline{2-8}
				& GRK & 64.5 & (44853, 1.2) & 288.5 & (214321, 1.1) & 75.5 & (70544, 1.4) \\
				& & 4.77 & (1.42) & 16.37 & (1.53) & 7.17 & (1.92) \\
				& & &5.70$\times 10^{-5}$  & &6.20$\times 10^{-3}$  & & 5.39$\times 10^{-2}$  \\
				\cline{2-8}
				& GK & 69 & (46389, 1.2) & 264 & (193503, 1.2) & 80 & (67099, 1.1) \\
				& & \color{red}$\star$1.91 & (0.38) & \color{red}$\star$5.94 & 
				(0.41) & \color{red}$\star$2.58 & (0.46) \\
				& & &4.28$\times 10^{-5}$  & &4.30$\times 10^{-3}$  & & 5.39$\times 10^{-2}$  \\
				\hline
			\end{tabular}
	\end{center}}
	{\footnotesize
		\begin{center}
			\begin{tablenotes}
				\item[1] {\color{black}{The CPU time in seconds for computing matrix $C$ is given below each name of the matrix.}}
				\item[2] First row: Number of outer iterations (total number of inner iterations, $\omega$).
				\item[3] Second row: Total CPU time, which includes parameter tuning 
				time in parentheses, in seconds.
				\item[4] Third row: Relative error norm.
			\end{tablenotes}
	\end{center}}
\end{table}

\begin{figure}[!t] 
	\centering
	\begin{minipage}[t]{0.46\linewidth}
		\centering
		\includegraphics[width=2.3in]{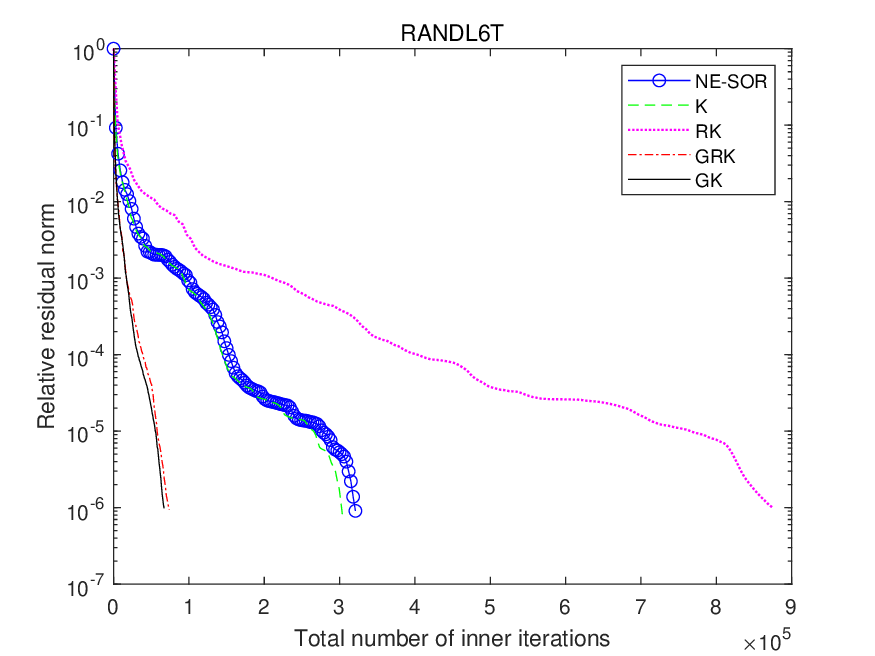}
	\end{minipage}
	\begin{minipage}[t]{0.46\linewidth}
		\centering
		\includegraphics[width=2.3in]{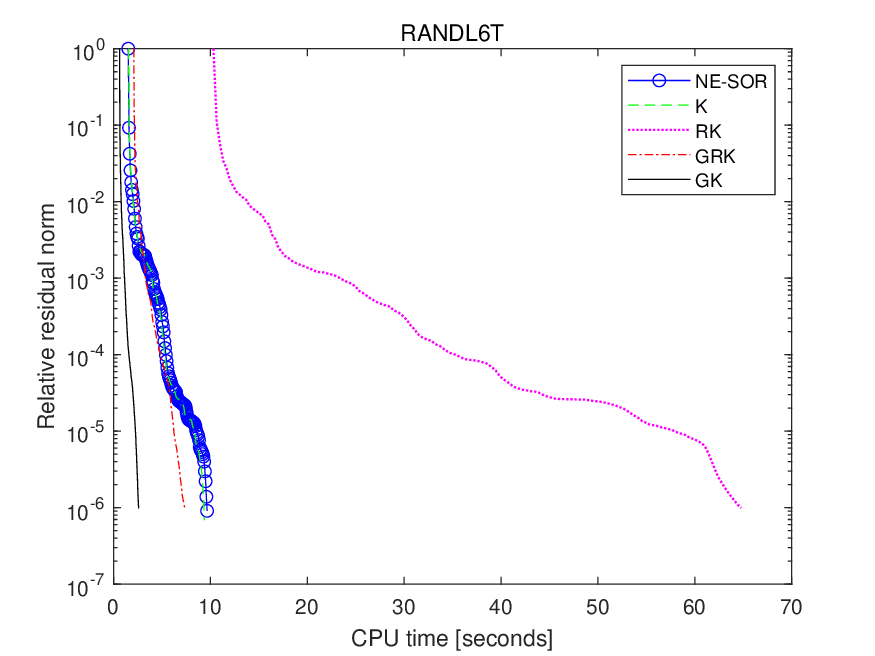}
	\end{minipage}
	\caption{Relative residual norm ${{{{\left\| {{{\boldsymbol{r}}_k}} \right\|}_2}}}/{{{{\left\| \boldsymbol{b} \right\|}_2}}}$ vs.\ total number of inner iterations (left) and relative residual norm ${{{{\left\| {{{\boldsymbol{r}}_k}} \right\|}_2}}}/{{{{\left\| \boldsymbol{b} \right\|}_2}}}$ vs.\ CPU time (right) for RANDL6T.}
	\label{fig:1}
\end{figure}

Table~\ref{tab:2} shows that F-AB-GMRES preconditioned by the GK inner iterations is the fastest among all the methods for the artificial random matrices RANDL$i$T, $i=1,2,\ldots,6$.
We remark that the total number of inner iterations of F-AB-GMRES with GRK and GK is also smaller than that of AB-GMRES with NE-SOR. This may imply that F-AB-GMRES with GRK and GK has a smaller workload than AB-GMRES with NE-SOR.

In Figure~\ref{fig:1}, we plot the relative residual norm ${{{{\left\| {{{\boldsymbol{r}}_k}} \right\|}_2}}}/{{{{\left\| \boldsymbol{b} \right\|}_2}}}$ versus {\color{black}the} total number of inner iterations and CPU time in seconds for the matrix RANDL6T.
Here, the CPU time includes the time for the parameter tuning and the computation of $C$. The figure shows that F-AB-GMRES preconditioned by the GK inner iterations is best among all the methods regarding the total number of inner iterations and CPU time for the matrix RANDL6T. These results are in accordance with Table~\ref{tab:2}.


\begin{table}[!t] 
	{\footnotesize
		\setlength{\tabcolsep}{3.8pt}
		\caption{Results for full-rank matrices (underdetermined problems).}\label{tab:fullrank_under}
		\begin{center}
			\begin{tabular}{ll|rrrrrr}
				\hline
				Outer & Inner & \multicolumn{2}{c}{illc1850T } & 
				\multicolumn{2}{c}{genT } & \multicolumn{2}{c}{photogrammetry2T } \\
				iteration & iteration & \multicolumn{2}{c}{0.01} & 
				\multicolumn{2}{c}{0.02} & \multicolumn{2}{c}{0.01} \\
				\hline
				AB-GMRES & NE-SOR & 262 & (746176, 1.1) & 35 & (430640, 
				0.8) & 27 & (151632, 0.9) \\
				& & \color{red}$\star$5.87 & (0.43) & 8.28 & (3.04) & 3.34 & (1.47) \\
				& & &2.72$\times 10^{-4}$  & &2.98$\times 10^{-6}$  & & 7.22$\times 10^{-2}$  \\
				\hline
				F-AB-GMRES & K & 281 & (731709, 1.1) & 48 & (408226, 0.9) & 26& (145494, 
				0.9) \\
				& & 6.17 & (0.41) & 7.39 & (2.18) & 3.28 & (1.47) \\
				& & &3.45$\times 10^{-4}$  & &2.87$\times 10^{-6}$  & & 7.22$\times 10^{-2}$  \\
				\cline{2-8}
				& RK & 396 & (2122830, 1.1) & 11 & (113731, 1.2) & 32 & (322140, 
				1.0) \\
				& & 95.29 & (6.60) & 20.84 & (15.13) & 31.62 & (15.09) \\
				& & &3.09$\times 10^{-4}$  & &8.94$\times 10^{-7}$  & & 7.22$\times 10^{-2}$  \\
				\cline{2-8}
				& GRK & 372 & (426620, 1.1) & 9 & (17815, 1.5) & 31 & (40500, 1.1) \\
				& & 23.55 & (1.74) & 4.61 & (3.53) & 4.84 & (2.31) \\
				& & &2.61$\times 10^{-4}$  & &5.45$\times 10^{-7}$  & & 7.22$\times 10^{-2}$  \\
				\cline{2-8}
				& GK & 593 & (666988, 1.2) & 8 & (15119, 1.6) & 32 & (49526, 1.1) \\
				& & 6.99 & (0.22) & \color{red}$\star$0.85 & 
				(0.60) & \color{red}$\star$1.22 & (0.45) \\
				& & &2.70$\times 10^{-4}$  & &6.97$\times 10^{-7}$  & & 7.22$\times 10^{-2}$ \\
				\hline
			\end{tabular}
	\end{center}}
	{\footnotesize
		\begin{center}
			\begin{tablenotes}
				\item[1] {\color{black}{The CPU time in seconds for computing matrix $C$ is given below each name of the matrix.}}
				\item[2] First row: Number of outer iterations (total number of inner iterations, $\omega$).
				\item[3] Second row: Total CPU time, which includes parameter tuning time in parentheses, in seconds.
				\item[4] Third row: Relative error norm.
			\end{tablenotes}
	\end{center}}
\end{table}

\begin{figure}[!t] 
	\centering
	\begin{minipage}[t]{0.46\linewidth}
		\centering
		\includegraphics[width=2.3in]{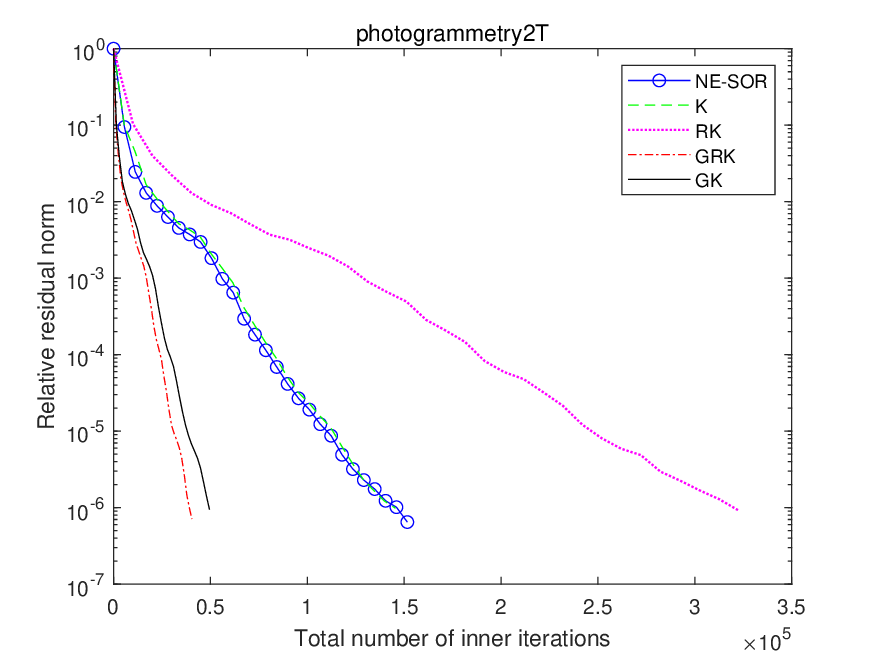}
	\end{minipage}
	\begin{minipage}[t]{0.46\linewidth}
		\centering
		\includegraphics[width=2.3in]{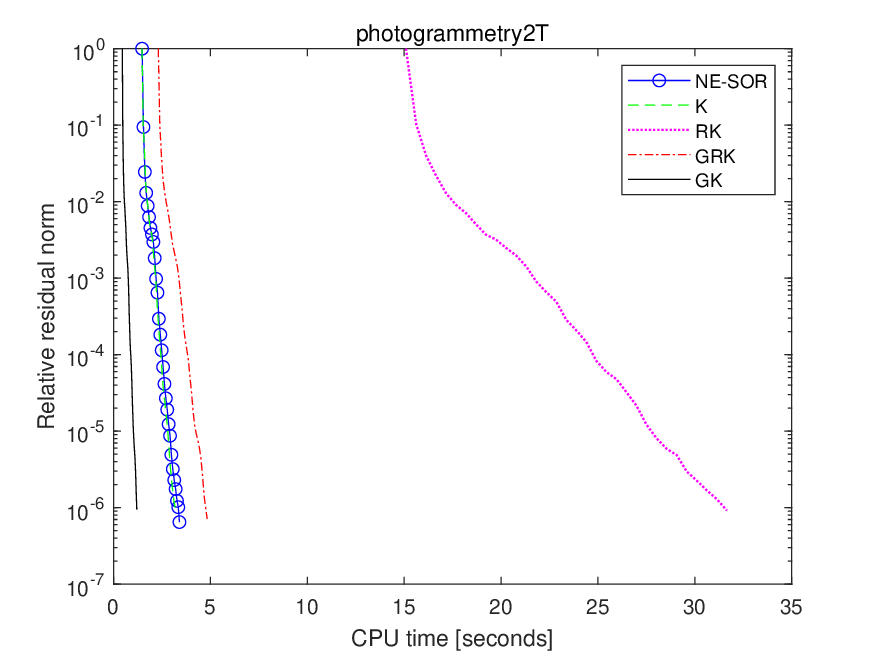}
	\end{minipage}
	\caption{Relative residual norm ${{{{\left\| {{{\boldsymbol{r}}_k}} \right\|}_2}}}/{{{{\left\| \boldsymbol{b} \right\|}_2}}}$ vs.\ total number of inner iterations (left) and relative residual norm ${{{{\left\| {{{\boldsymbol{r}}_k}} \right\|}_2}}}/{{{{\left\| \boldsymbol{b} \right\|}_2}}}$ vs.\ CPU time (right) for photogrammetry2T.}
	\label{fig:fullrank_under}
\end{figure}


{\color{black}{Table~\ref{tab:fullrank_under} shows that AB-GMRES preconditioned by the NE-SOR inner iterations is the fastest among all the methods for the matrix illc1850T. F-AB-GMRES preconditioned by the GK inner iterations is the fastest among all the methods for the matrices genT and photogrammetry2T.}}

In Figure~\ref{fig:fullrank_under}, we plot the relative residual norm ${{{{\left\| {{{\boldsymbol{r}}_k}} \right\|}_2}}}/{{{{\left\| \boldsymbol{b} \right\|}_2}}}$ versus the total number of inner iterations and CPU time in seconds for the matrix photogrammetry2T. Figure~\ref{fig:fullrank_under} shows that F-AB-GMRES preconditioned by the GRK inner iterations is best among all the methods when comparing the total number of inner iterations, and the GK inner iterations is best among all the methods regarding the CPU time for the matrix photogrammetry2T. These results are in accordance with Table~\ref{tab:fullrank_under}.

\begin{table}[!t] 
	{\footnotesize
		\setlength{\tabcolsep}{3.7pt}
		\caption{Results for rank-deficient matrices (underdetermined problems).}\label{tab:3}
		\begin{center}
			\begin{tabular}{ll|rrrrrr}
				\hline
				Outer & Inner & \multicolumn{2}{c}{Maragal$\_$3T } & 
				\multicolumn{2}{c}{Maragal$\_$4T } & \multicolumn{2}{c}{Maragal$\_$5T } \\
				iteration & iteration & \multicolumn{2}{c}{0.01} & 
				\multicolumn{2}{c}{0.02} & \multicolumn{2}{c}{0.19} \\
				\hline
				AB-GMRES & NE-SOR &  57 & (195624, 1.2) & 51 & (209508, 
				1.1) & 144 & (1898496, 1.2) \\
				& & 2.62 & (0.66) & 3.39 & (0.97) & 43.12 & (5.29) \\
				& & &3.02$\times 10^{-5}$  & &2.97$\times 10^{-2}$  & & 1.71$\times 10^{-4}$  \\
				\hline
				F-AB-GMRES & K & 68 & (176392, 1.0) & 60 & (186660, 1.1) & 176 & (1747326, 
				1.1) \\
				& & 2.39 & (0.53) & 3.06 & (0.76) & 41.66 & (4.29) \\
				& & &4.88$\times 10^{-5}$  & &2.97$\times 10^{-2}$  & & 1.71$\times 10^{-4}$  \\
				\cline{2-8}
				& RK & 284.5 & (1402900, 1.1) &168 & (1124400, 1.1) & 498.5 & (8551269, 
				1.1) \\
				& & 73.81 & (6.93) & 66.73 & (9.93) & 1057.50 & (54.89) \\
				& & &7.38$\times 10^{-6}$  & &2.97$\times 10^{-2}$  & & 2.60$\times 10^{-3}$  \\
				\cline{2-8}
				& GRK & 67.5 & (44395, 1.2) & 54 & (56203, 1.2) & 139 & (340070, 1.1) \\
				& & 3.61 & (1.10) & 5.33 & (1.90) & 62.12 & (10.74) \\
				& & &1.61$\times 10^{-5}$  & &2.97$\times 10^{-2}$  & & 2.96$\times 10^{-4}$  \\
				\cline{2-8}
				& GK & 129 & (83850, 1.0) & 86 & (81420, 1.1) & 219 & (550274, 1.1) \\
				& & \color{red}$\star$1.28 & (0.18) & \color{red}$\star$1.53 & 
				(0.30) & \color{red}$\star$16.98 & (1.49) \\
				& & &3.53$\times 10^{-5}$  & &2.97$\times 10^{-2}$  & & 7.55$\times 10^{-4}$  \\
				\hline
			\end{tabular}
	\end{center}}
	{\footnotesize
		\begin{center}
			\begin{tablenotes}
				\item[1] {\color{black}{The CPU time in seconds for computing matrix $C$ is given below each name of the matrix.}}
				\item[2] First row: Number of outer iterations (total number of inner iterations, $\omega$).
				\item[3] Second row: Total CPU time, which includes parameter tuning time in parentheses, in seconds.
				\item[4] Third row: Relative error norm.
			\end{tablenotes}
	\end{center}}
\end{table}

Table~\ref{tab:3} shows that F-AB-GMRES preconditioned by the GK inner iterations is also the fastest among all the methods for the rank-deficient matrices Maragal$\_$3T, Maragal$\_$4T and Maragal$\_$5T.

In Figure~\ref{fig:2}, we plot the relative residual norm ${{{{\left\| {{{\boldsymbol{r}}_k}} \right\|}_2}}}/{{{{\left\| \boldsymbol{b} \right\|}_2}}}$ versus {\color{black}the} total number of inner iterations and CPU time in seconds for the matrix Maragal$\_$5T.
Figure~\ref{fig:2} shows that F-AB-GMRES preconditioned by the GRK inner iterations is best among all the methods regarding the total number of inner iterations, and the GK inner iterations is the fastest among all the methods regarding CPU time for the matrix Maragal$\_$5T. These results are in accordance with Table~\ref{tab:3}.

\begin{figure}[!t] 
	\centering
	\begin{minipage}[t]{0.46\linewidth}
		\centering
		\includegraphics[width=2.3in]{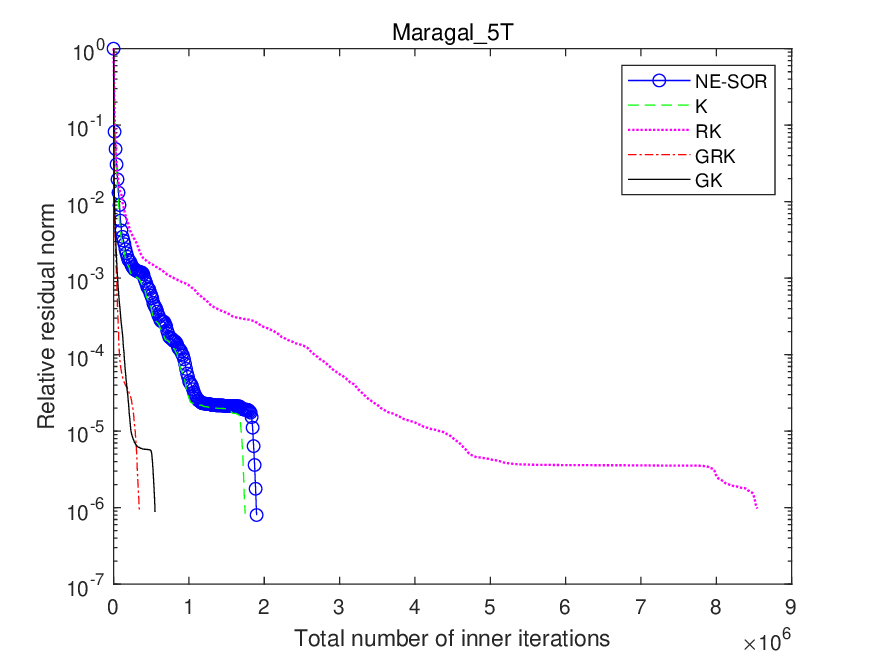}
	\end{minipage}
	\begin{minipage}[t]{0.46\linewidth}
		\centering
		\includegraphics[width=2.3in]{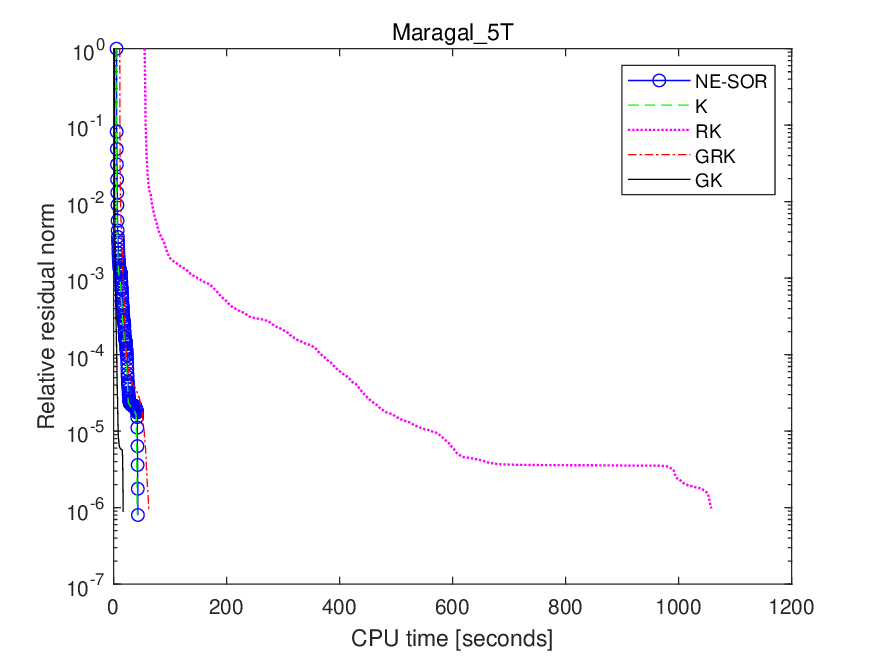}
	\end{minipage}
	\caption{Relative residual norm ${{{{\left\| {{{\boldsymbol{r}}_k}} \right\|}_2}}}/{{{{\left\| \boldsymbol{b} \right\|}_2}}}$ vs.\ total number of inner iterations (left) and relative residual norm ${{{{\left\| {{{\boldsymbol{r}}_k}} \right\|}_2}}}/{{{{\left\| \boldsymbol{b} \right\|}_2}}}$ vs.\ CPU time (right) for Maragal$\_$5T.}
	\label{fig:2}
\end{figure}

\subsection{Overdetermined problems}
Next, we present numerical experiment results on overdetermined problems ($m>n$). 
Tables~\ref{tab:4},~\ref{tab:fullrank_over} and~\ref{tab:5} give the numerical experiment results for artificial random matrices, full-rank matrices and rank-deficient matrices, respectively similarly to Tables~\ref{tab:2},~\ref{tab:fullrank_under} and~\ref{tab:3}. 


\begin{table}[!t] 
	{\footnotesize
		\setlength{\tabcolsep}{4.0pt}
		\caption{Results for full-rank artificial random matrices (overdetermined problems).}\label{tab:4}
		\begin{center}
			\begin{tabular}{ll|rrrrrr}
				\hline
				Outer & Inner & \multicolumn{2}{c}{RANDL1} & 
				\multicolumn{2}{c}{RANDL2} & \multicolumn{2}{c}{RANDL3} \\
				iteration & iteration & \multicolumn{2}{c}{2.22} & 
				\multicolumn{2}{c}{2.27} & \multicolumn{2}{c}{2.17} \\
				\hline
				AB-GMRES & NE-SOR &  2 & (10000, 1.0) & 2 & (10000, 1.1) & 2 & (10000, 0.7) \\
				& & 3.87 & (3.52) & 3.87 & (3.51) &3.89 & (3.53) \\
				& & &1.94$\times 10^{-8}$  & &9.54$\times 10^{-7}$  & & 1.68$\times 10^{-5}$  \\
				\hline
				F-AB-GMRES & K & $--$ & ($--$, 1.0) & $--$ & ($--$, 1.0) & $--$ & ($--$, 
				0.9) \\
				& & $--$ & (1.25) & $--$ & (1.17) & $--$ & (1.25) \\
				& & &$--$  & &$--$  & & $--$  \\
				\cline{2-8}
				& RK & 9 & (34049, 1.0) & 165 & (935140, 1.0) & 256 & (1362812, 
				1.0) \\
				& & 25.57 & (17.88) & 165.87 & (21.55) & 231.57 & (20.70) \\
				& & &8.49$\times 10^{-7}$  & &6.53$\times 10^{-6}$  & & 3.61$\times 10^{-6}$  \\
				\cline{2-8}
				& GRK & 6 & (1696.5, 1.0) & 6 & (1768, 1.0) & 6 & (1763, 1.0) \\
				& & 4.61 & (2.06) & 4.73 & (2.11) & 4.57 & (2.04) \\
				& & &8.85$\times 10^{-7}$  & &8.15$\times 10^{-7}$  & & 8.88$\times 10^{-7}$  \\
				\cline{2-8}
				& GK & 7 & (2472, 1.0) & 16 & (5904, 1.0) & 28 & (8568, 1.0) \\
				& & \color{red}$\star$2.66 & (0.33) & \color{red}$\star$2.85 & 
				(0.32) &  \color{red}$\star$2.83 & (0.29) \\
				& & &2.99$\times 10^{-7}$  & &2.75$\times 10^{-6}$  & & 9.91$\times 10^{-6}$ \\
				\hline
			\end{tabular}
	\end{center}}
	\medskip
	{\footnotesize
		\setlength{\tabcolsep}{3.2pt}
		\begin{center}
			\begin{tabular}{ll|rrrrrr}
				\hline
				Outer & Inner & \multicolumn{2}{c}{RANDL4} & 
				\multicolumn{2}{c}{RANDL5} & \multicolumn{2}{c}{RANDL6} \\
				iteration & iteration & \multicolumn{2}{c}{2.26} & 
				\multicolumn{2}{c}{2.32} & \multicolumn{2}{c}{2.24} \\
				\hline
				AB-GMRES & NE-SOR & 2 & (10000, 1.0)& 2 & (10000, 
				1.0) & 3 & (15000, 0.7) \\
				& & 3.94 & (3.58) &3.97 & (3.61) & 4.07 & (3.53) \\
				& & &3.47$\times 10^{-8}$ & &2.54$\times 10^{-9}$  & & 4.39$\times 10^{-7}$  \\
				\hline
				F-AB-GMRES & K & $--$ & ($--$, 1.0) & $--$ & ($--$, 1.0) & $--$ & 
				($--$, 0.9) \\
				& & $--$ & (1.07) & $--$ & (1.33) & $--$ & (1.12) \\
				& & &$--$  & &$--$  & & $--$  \\
				\cline{2-8}
				& RK & 232 & (1223500, 1.0) & 293.5 & (1538400, 1.0) & 258.5 & 
				(1307200, 1.0) \\
				& & 209.96 & (20.25) & 259.75 & (22.10) & 223.15 & (20.85) \\
				& & &2.81$\times 10^{-5}$ & &1.13$\times 10^{-2}$  & & 2.05$\times 10^{-2}$  \\
				\cline{2-8}
				& GRK & 6 & (1710.5, 1.0) & 6 & (1761, 1.0) & 6 & (1662.5, 1.0) \\
				& & 4.57 & (1.97) & 4.88 & (2.20) & 4.58 & (2.00) \\
				& & &9.44$\times 10^{-7}$  & &8.67$\times 10^{-7}$  & & 8.74$\times 10^{-7}$  \\
				\cline{2-8}
				& GK & 22 & (7282, 1.0) & 36 & (13104, 1.0) & 25 & (8350, 1.0) \\
				& & \color{red}$\star$2.88 & (0.31) & \color{red}$\star$3.20 & 
				(0.34) & \color{red}$\star$2.91 & (0.31) \\
				& & &1.51$\times 10^{-4}$  & &9.32$\times 10^{-4}$  & & 2.05$\times 10^{-2}$ \\
				\hline
			\end{tabular}
	\end{center}}
	{\footnotesize
		\begin{center}
			\begin{tablenotes}
				\item[1] {\color{black}{The CPU time in seconds for computing matrix $C$ is given below each name of the matrix.}}
				\item[2] First row: Number of outer iterations (total number of inner iterations, $\omega$).
				\item[3] Second row: Total CPU time, which includes parameter tuning time in parentheses, in seconds.
				\item[4] Third row: Relative error norm.
			\end{tablenotes}
	\end{center}}
\end{table}

{\color{black}{Table~\ref{tab:4} shows that F-AB-GMRES preconditioned by the GK inner iterations is the fastest among all the methods for the artificial random matrices RANDL$i$, $i=1,2,\ldots,6$.}}
We should remark that the time for actual execution for NE-SOR method is small, but the time for parameter tuning is large. 



\begin{table}[!t] 
	{\footnotesize
		\setlength{\tabcolsep}{3.5pt}
		\caption{Results for full-rank matrices (overdetermined problems).}\label{tab:fullrank_over}
		\begin{center}
			\begin{tabular}{ll|rrrrrr}
				\hline
				Outer & Inner & \multicolumn{2}{c}{illc1850 } & 
				\multicolumn{2}{c}{gen} & \multicolumn{2}{c}{ photogrammetry2} \\
				iteration & iteration & \multicolumn{2}{c}{0.03} & 
				\multicolumn{2}{c}{0.09} & \multicolumn{2}{c}{0.02} \\
				\hline
				AB-GMRES & NE-SOR &  260 & (1443000, 0.9) & 46 & (589030, 
				0.4) & 20 & (536640, 0.8) \\
				& & 12.92 & (0.93) & 11.45 & (3.49) & 11.98 & (6.03) \\
				& & &1.35$\times 10^{-4}$  & &1.92$\times 10^{-6}$  & & 5.86$\times 10^{-2}$  \\
				\hline
				F-AB-GMRES & K & 266 & (1312908, 0.9) & 51 & (586961, 0.4) & 23 & (525370, 
				0.7) \\
				& & 13.06 & (0.91) & 12.04 & (3.38) & 12.59 & (5.85) \\
				& & &1.40$\times 10^{-4}$  & &2.51$\times 10^{-6}$  & & 5.86$\times 10^{-2}$  \\
				\cline{2-8}
				& RK & 419& (2161000, 1.1) & 13 & (130256, 1.1) & 29.5 & (239953, 
				1.2) \\
				& & 125.69 & (8.09) & 30.47 & (21.00) & 49.78 & (22.56) \\
				& & &3.37$\times 10^{-4}$ & &2.17$\times 10^{-6}$  & & 5.86$\times 10^{-2}$  \\
				\cline{2-8}
				& GRK & 409.5 & (364010, 1.1) & 6 & (4558, 1.3) & 20 & (25409, 1.4) \\
				& & 26.98 & (1.84) & 2.81 & (2.28) & 9.77 & (5.55) \\
				& & &1.17$\times 10^{-4}$  & &1.26$\times 10^{-6}$  & & 5.86$\times 10^{-2}$  \\
				\cline{2-8}
				& GK & 413 & (381596, 1.1) & 10 & (14899, 1.1) & 21 & (23718, 1.3) \\
				& & \color{red}$\star$5.21 & (0.25) & \color{red}$\star$0.95 & 
				(0.57) & \color{red}$\star$0.91 & (0.43) \\
				& & &2.52$\times 10^{-4}$ & &2.19$\times 10^{-6}$  & & 5.86$\times 10^{-2}$  \\
				\hline
			\end{tabular}
	\end{center}}
	{\footnotesize
		\begin{center}
			\begin{tablenotes}
				\item[1] {\color{black}{The CPU time in seconds for computing matrix $C$ is given below each name of the matrix.}}
				\item[2] First row: Number of outer iterations (total number of inner iterations, $\omega$).
				\item[3] Second row: Total CPU time, which includes parameter tuning time in parentheses, in seconds.
				\item[4] Third row: Relative error norm.
			\end{tablenotes}
	\end{center}}
\end{table}

Table~\ref{tab:fullrank_over} shows that F-AB-GMRES preconditioned by the GK method is the fastest regarding the CPU time among all the methods for matrices illc1850, gen and photogrammetry2.

In Figure~\ref{fig:fullrank_over}, we plot the relative residual norm ${{{{\left\| {{{\boldsymbol{r}}_k}} \right\|}_2}}}/{{{{\left\| \boldsymbol{b} \right\|}_2}}}$ versus the total number of inner iterations and CPU time in seconds for the matrix photogrammetry2. 
{\color{black}{Figure~\ref{fig:fullrank_over} shows that F-AB-GMRES preconditioned by the GK inner iterations is best among all the methods regarding the total number of inner iterations and CPU time for the matrix photogrammetry2.}}
These results are in accordance with Table~\ref{tab:fullrank_over}.

\begin{figure}[!t] 
	\centering
	\begin{minipage}[t]{0.46\linewidth}
		\centering
		\includegraphics[width=2.3in]{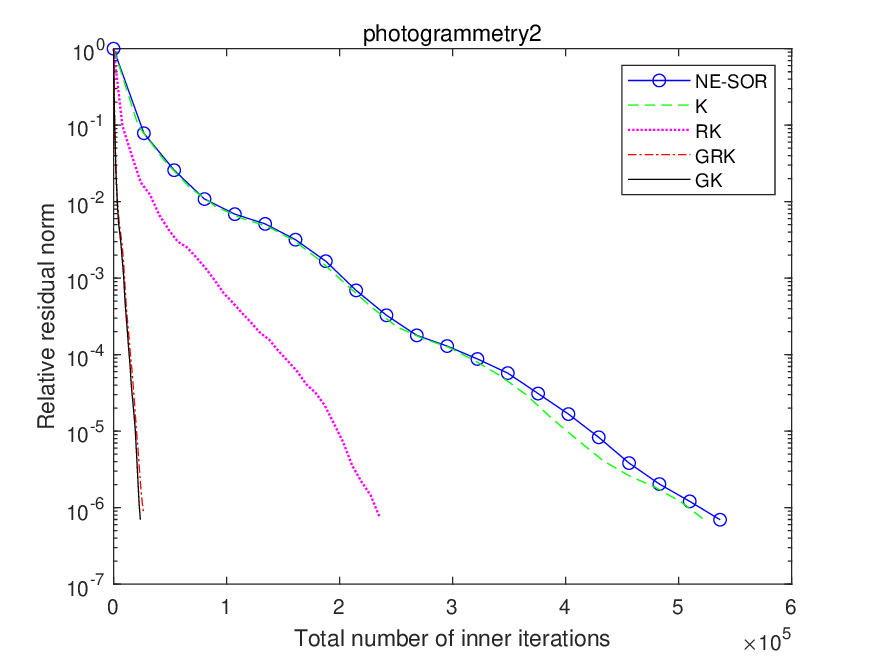}
	\end{minipage}
	\begin{minipage}[t]{0.46\linewidth}
		\centering
		\includegraphics[width=2.3in]{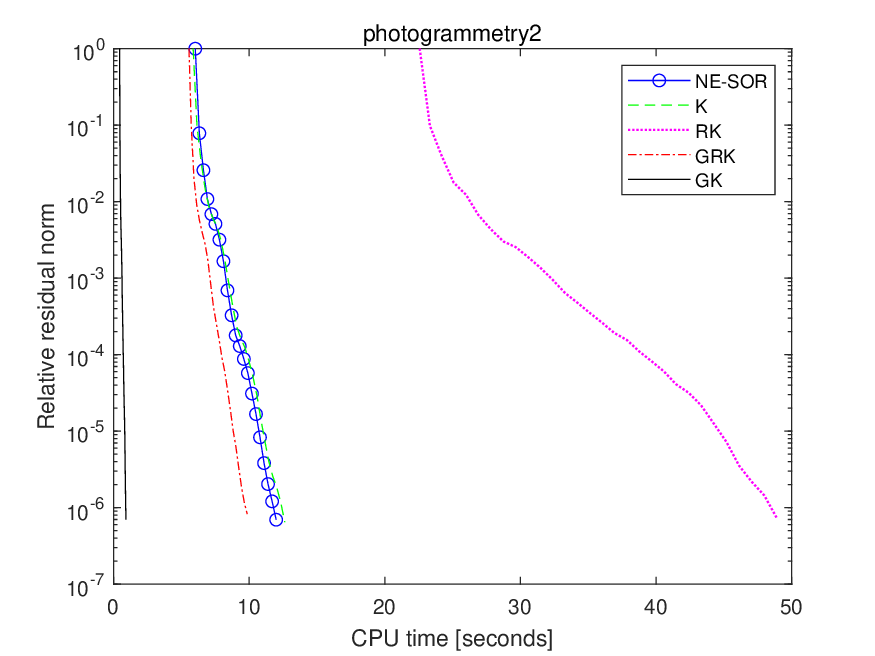}
	\end{minipage}
	\caption{Relative residual norm ${{{{\left\| {{{\boldsymbol{r}}_k}} \right\|}_2}}}/{{{{\left\| \boldsymbol{b} \right\|}_2}}}$ vs.\ total number of inner iterations (left) and relative residual norm ${{{{\left\| {{{\boldsymbol{r}}_k}} \right\|}_2}}}/{{{{\left\| \boldsymbol{b} \right\|}_2}}}$ vs.\ CPU time (right) for photogrammetry2.}
	\label{fig:fullrank_over}
\end{figure}

Table~\ref{tab:5} shows that F-AB-GMRES preconditioned by
the GK method is the fastest regarding the CPU time among all the methods for matrices Maragal$\_$3, Maragal$\_$4 and Maragal$\_$5.

\begin{table}[!t] 
	{\footnotesize
		\setlength{\tabcolsep}{3.5pt}
		\caption{Results for rank-deficient matrices (overdetermined problems).}\label{tab:5}
		\begin{center}
			\begin{tabular}{ll|rrrrrr}
				\hline
				Outer & Inner & \multicolumn{2}{c}{Maragal$\_$3 } & 
				\multicolumn{2}{c}{Maragal$\_$4 } & \multicolumn{2}{c}{Maragal$\_$5 } \\
				iteration & iteration & \multicolumn{2}{c}{0.02} & 
				\multicolumn{2}{c}{0.02} & \multicolumn{2}{c}{0.06} \\
				\hline
				AB-GMRES & NE-SOR &  149 & (751854, 1.1) & 97 & (571524, 
				1.1) & 330 & (4607460, 1.1) \\
				& & 8.35 & (1.03) & 6.80 & (1.18) & 81.49 & (4.69) \\
				& & &1.51$\times 10^{-5}$  & &1.59$\times 10^{-2}$  & & 9.14$\times 10^{-4}$ \\
				\hline
				F-AB-GMRES & K & 172 & (630654, 1.1) & 127 & (495173, 1.0) & 397 & (4115302, 
				1.0) \\
				& & 7.54 & (0.78) & 6.19 & (0.83) & 81.34 & (3.85) \\
				& & &4.94$\times 10^{-5}$  & &1.59$\times 10^{-2}$  & & 9.14$\times 10^{-4}$  \\
				\cline{2-8}
				& RK & 243.5 & (1152207, 1.1) & 143.5 & (870448, 1.1) & 466.5 & (8306272, 
				1.1) \\
				& & 70.53 & (7.50) & 59.16 & (9.95) & 1119.70 & (62.89) \\
				& & &7.12$\times 10^{-5}$  & &1.59$\times 10^{-2}$  & & 1.55$\times 10^{-2}$  \\
				\cline{2-8}
				& GRK & 231 & (125664, 1.1) &144.5 & (114100, 1.1) & 475.5 & (986660, 1.0) \\
				& & 9.71 & (1.10) &9.92 & (1.66) & 184.85 & (10.45) \\
				& & &6.10$\times 10^{-5}$  & &1.59$\times 10^{-2}$  & & 7.10$\times 10^{-3}$ \\
				\cline{2-8}
				& GK & 250 & (159496, 1.0) & 136 & (116784, 1.1) & 464 & (1085222, 1.1) \\
				& & \color{red}$\star$2.58 & (0.19) & \color{red}$\star$2.06 & 
				(0.26) & \color{red}$\star$33.25 & (1.29) \\
				& & &6.21$\times 10^{-5}$  & &1.59$\times 10^{-2}$  & & 3.70$\times 10^{-3}$  \\
				\hline
			\end{tabular}
	\end{center}}
	{\footnotesize
		\begin{center}
			\begin{tablenotes}
				\item[1] {\color{black}{The CPU time in seconds for computing matrix $C$ is given below each name of the matrix.}}
				\item[2] First row: Number of outer iterations (total number of inner iterations, $\omega$).
				\item[3] Second row: Total CPU time, which includes parameter tuning time in parentheses, in seconds.
				\item[4] Third row: Relative error norm.
			\end{tablenotes}
	\end{center}}
\end{table}

\begin{figure}[!t] 
	\centering
	\begin{minipage}[t]{0.46\linewidth}
		\centering
		\includegraphics[width=2.3in]{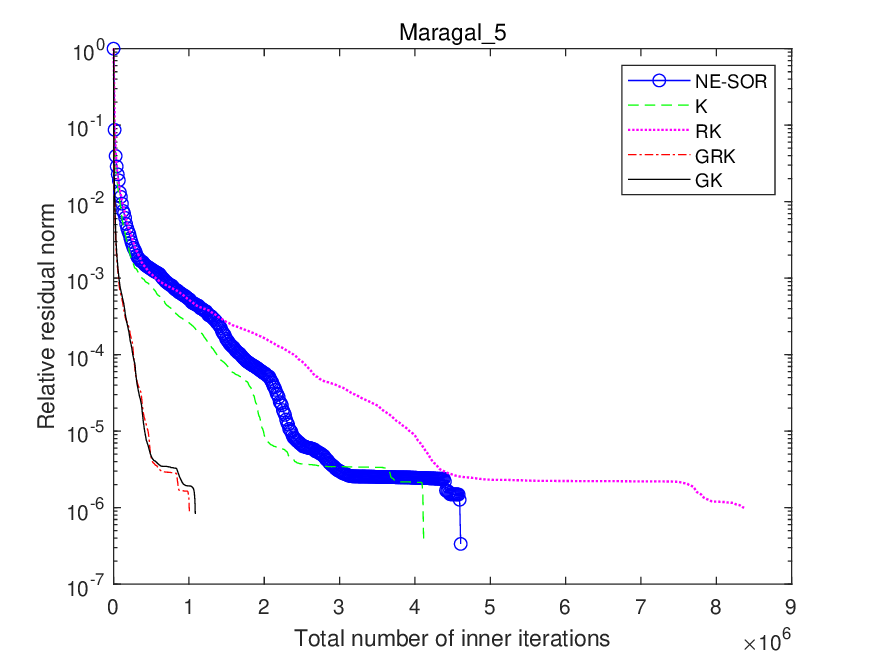}
	\end{minipage}
	\begin{minipage}[t]{0.46\linewidth}
		\centering
		\includegraphics[width=2.3in]{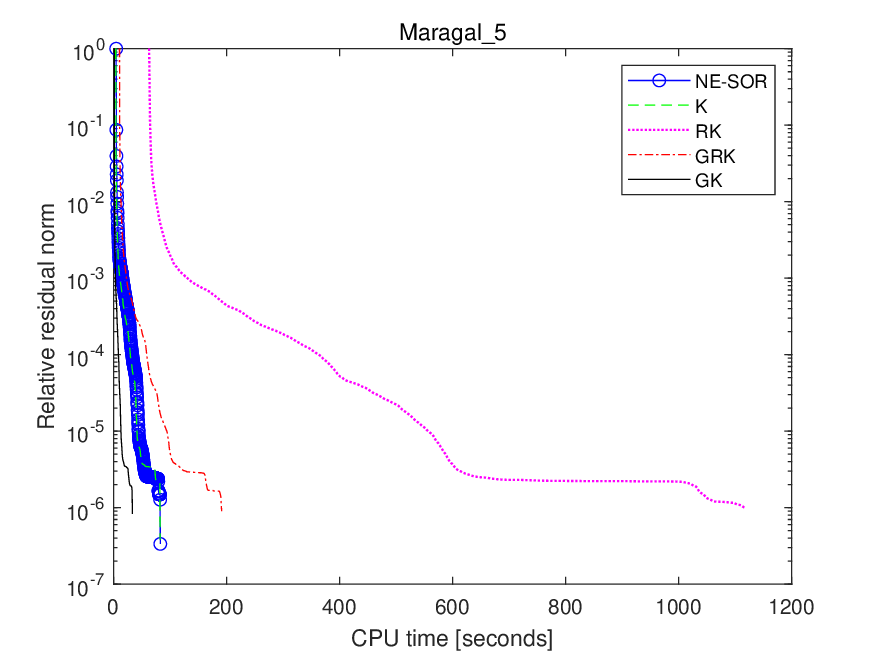}
	\end{minipage}
	\caption{Relative residual norm ${{{{\left\| {{{\boldsymbol{r}}_k}} \right\|}_2}}}/{{{{\left\| \boldsymbol{b} \right\|}_2}}}$ vs.\ total number of inner iterations (left) and relative residual norm ${{{{\left\| {{{\boldsymbol{r}}_k}} \right\|}_2}}}/{{{{\left\| \boldsymbol{b} \right\|}_2}}}$ vs.\ CPU time (right) for Maragal$\_$5.}
	\label{fig:4}
\end{figure}

In Figure~\ref{fig:4}, we plot the relative residual norm ${{{{\left\| {{{\boldsymbol{r}}_k}} \right\|}_2}}}/{{{{\left\| \boldsymbol{b} \right\|}_2}}}$ versus the total number of inner iterations and CPU time in seconds for the matrix Maragal$\_$5. Figure~\ref{fig:4} shows that F-AB-GMRES preconditioned by the GRK inner iterations is best among all the methods when comparing the total number of inner iterations, and the GK inner iterations is the fastest among all the methods regarding the CPU time for the matrix Maragal$\_$5. These results are in accordance with Table~\ref{tab:5}.

We have tried to further speed up the methods based on the Kaczmarz-type inner iterations by computing approximations of $AA^\mathsf{T}$ using the method in~\cite{HI15} for over- and underdetermined systems, but so far we have not been successful, and this is left for future research.

\subsection{Inconsistent problems}

In order to test our method for inconsistent problems, we 
{\color{black}{let $\boldsymbol{b}=A\boldsymbol{x}_{\star}$ and add noise to $\boldsymbol{b}$ to obtain the right-hand side $\tilde{\boldsymbol{b}}=[\tilde{b}_1,\ldots,\tilde{b}_m]^\mathsf{T}$ by letting
		$$\tilde{b}_i = b_i \cdot (1 + \epsilon \cdot \mu_i), \quad i=1,\ldots,m.$$
		Different values $\epsilon=10^{-3}$, $10^{-2}$ and $10^{-1}$ were used for $\epsilon$.
		The scalars $\mu_i~(i=1,\ldots,m)$ were generated randomly in the interval (-1,1) using the MATLAB function \texttt{rand}.}}

\begin{figure}[!t] 
	\centering
	\begin{minipage}[t]{0.46\linewidth}
		\centering
		\includegraphics[width=2.3in]{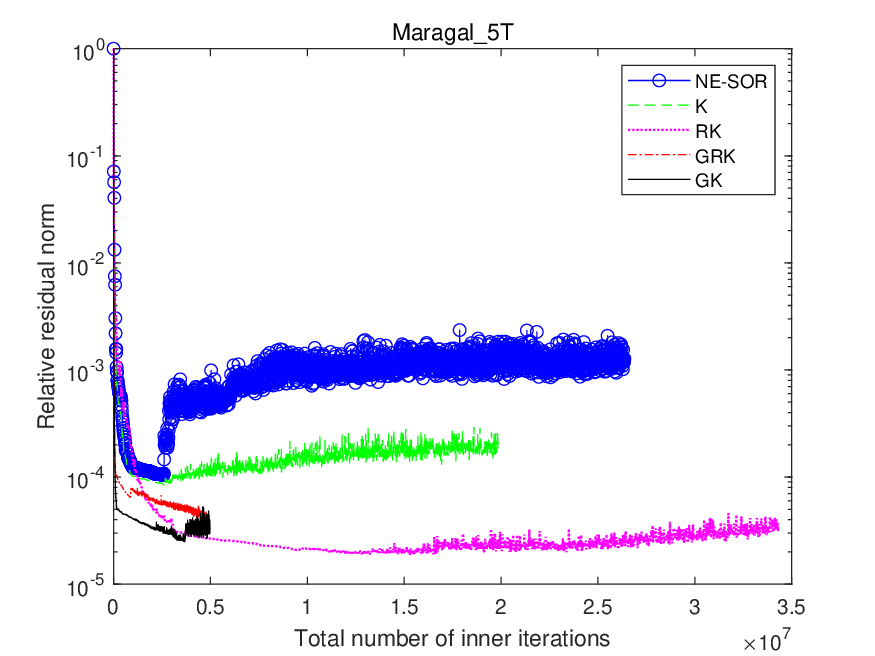}
	\end{minipage}
	\begin{minipage}[t]{0.46\linewidth}
		\centering
		\includegraphics[width=2.3in]{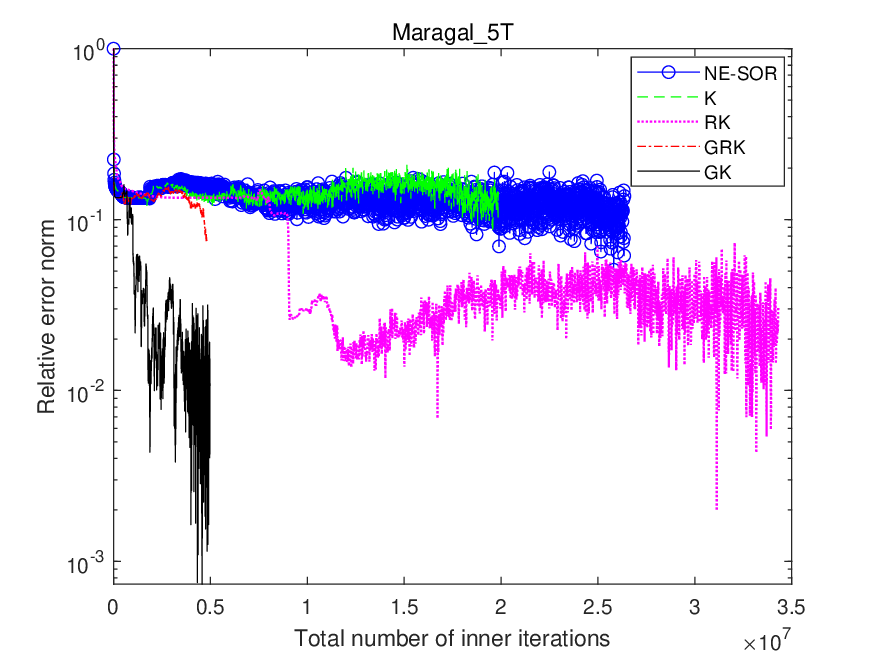}
	\end{minipage}
	\caption{Relative residual norm ${{{{\| {{{A^\mathsf{T}\boldsymbol{r}}_k}} \|}_2}}}/{{{{\| A^\mathsf{T}{\tilde{\boldsymbol{b}}} \|}_2}}}$ vs.\ total number of inner iterations (left) and relative error norm ${\left\| {{\boldsymbol{x}_k} - {\boldsymbol{x}_ *  }} \right\|_2} / {\left\| {{\boldsymbol{x}_ * }} \right\|_2}$ vs.\ total number of inner iterations (right) for Maragal$\_$5T for $\epsilon=10^{-3}$.}
	\label{fig:5}
\end{figure}

\begin{figure}[!t] 
	\centering
	\begin{minipage}[t]{0.46\linewidth}
		\centering
		\includegraphics[width=2.3in]{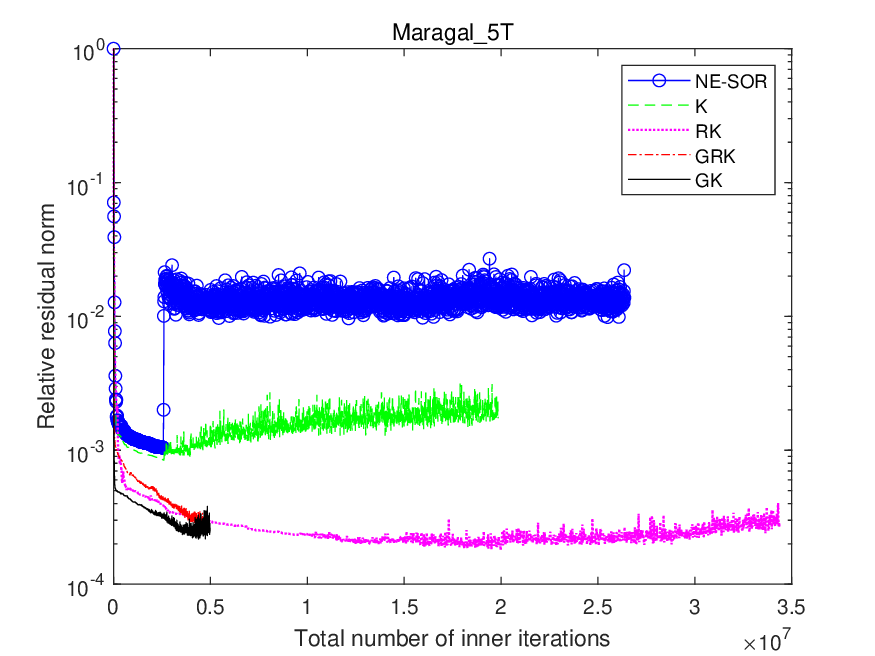}
	\end{minipage}
	\begin{minipage}[t]{0.46\linewidth}
		\centering
		\includegraphics[width=2.3in]{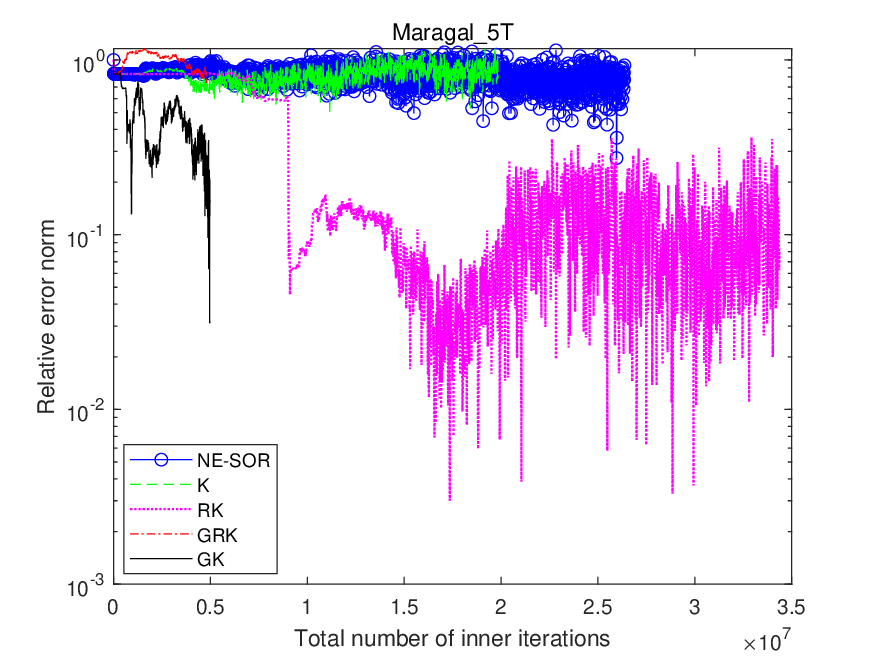}
	\end{minipage}
	\caption{Relative residual norm ${{{{\| {{{A^\mathsf{T}\boldsymbol{r}}_k}} \|}_2}}}/{{{{\| A^\mathsf{T}{\tilde{\boldsymbol{b}}} \|}_2}}}$ vs.\ total number of inner iterations (left) and relative error norm ${\left\| {{\boldsymbol{x}_k} - {\boldsymbol{x}_ *  }} \right\|_2} / {\left\| {{\boldsymbol{x}_ * }} \right\|_2}$ vs.\ total number of inner iterations (right) for Maragal$\_$5T for $\epsilon=10^{-2}$.}
	\label{fig:6}
\end{figure}

\begin{figure}[!t] 
	\centering
	\begin{minipage}[t]{0.46\linewidth}
		\centering
		\includegraphics[width=2.3in]{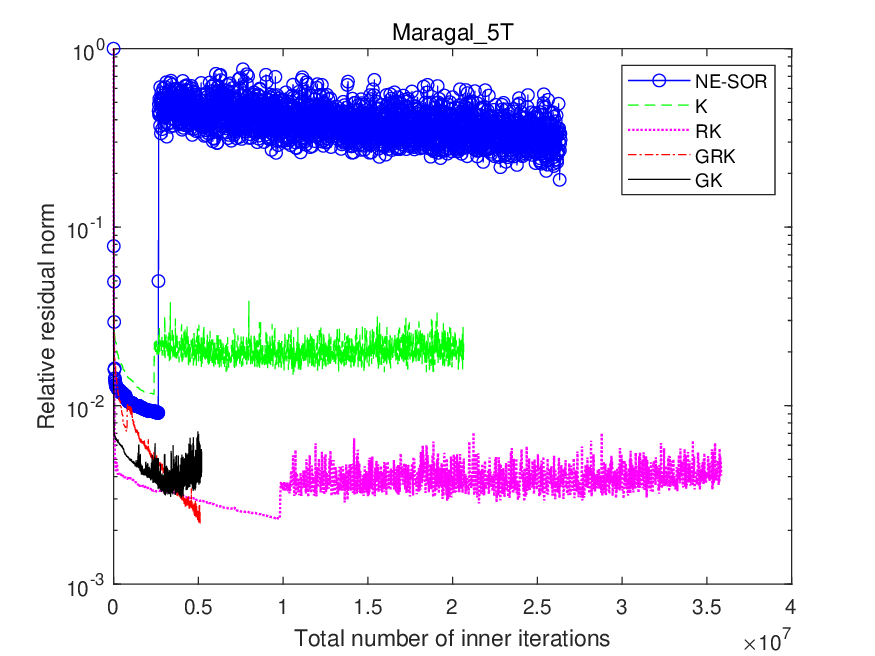}
	\end{minipage}
	\begin{minipage}[t]{0.46\linewidth}
		\centering
		\includegraphics[width=2.3in]{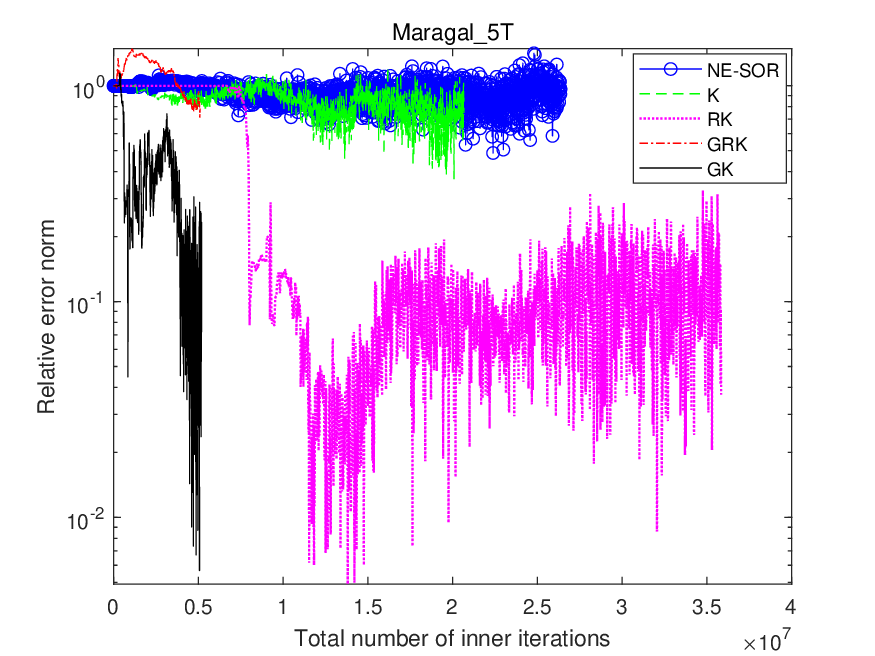}
	\end{minipage}
	\caption{Relative residual norm ${{{{\| {{{A^\mathsf{T}\boldsymbol{r}}_k}} \|}_2}}}/{{{{\| A^\mathsf{T}{\tilde{\boldsymbol{b}}} \|}_2}}}$ vs.\ total number of inner iterations (left) and relative error norm ${\left\| {{\boldsymbol{x}_k} - {\boldsymbol{x}_ *  }} \right\|_2} / {\left\| {{\boldsymbol{x}_ * }} \right\|_2}$ vs.\ total number of inner iterations (right) for Maragal$\_$5T for $\epsilon=10^{-1}$.}
	\label{fig:7}
\end{figure}

In Figure{\color{black}s}~\ref{fig:5}{\color{black},\ref{fig:6} and~\ref{fig:7}}, we plot the relative residual norm ${{{{\| {{{A^\mathsf{T}\boldsymbol{r}}_k}} \|}_2}}}/{{{{\| A^\mathsf{T}{\color{black}{\tilde{\boldsymbol{b}}}} \|}_2}}}$ for the normal equations {\color{black}{where $\boldsymbol{r}_k=\tilde{\boldsymbol{b}}-A{\boldsymbol{x}_k}$}} and relative error norm ${\left\| {{\boldsymbol{x}_k} - {\boldsymbol{x}_ *  }} \right\|_2} / {\left\| {{\boldsymbol{x}_ * }} \right\|_2}$ {\color{black}{where $x_*=A^\dag \tilde{\boldsymbol{b}}$}} versus the total number of inner iterations for the matrix Maragal$\_$5T {\color{black}for $\epsilon=10^{-3}$, $10^{-2}$ and $10^{-1}$,} repectively. Figure{\color{black}s}~\ref{fig:5}{\color{black},\ref{fig:6} and~\ref{fig:7}} show that all the methods {\color{black}do} not converge for this inconsistent problem and {\color{black}as the value of $\epsilon$ increases, the smallest residual norm 
	increases for each method.}

\section{Conclusion}
In this paper, we proposed replacing the NE-SOR method by Kaczmarz-type methods in the previous AB-GMRES method preconditioned by stationary inner iterations for solving consistent systems of linear equations.
To do so, we developed a new algorithm called flexible AB-GMRES method preconditioned by Kaczmarz-type methods as inner iterations.
An optimality property of minimizing residuals was given for the proposed method.
We also proposed a tuning procedure for adjusting the maximum number of inner iterations and value of the relaxation parameter in the method.
Numerical experiment results showed that flexible AB-GMRES preconditioned by Kaczmarz-type methods converge faster than the previous method in terms of total CPU time.

\section*{Appendix A} \textbf{Comparison of computational work of GK and modified GK inner-iteration preconditioning.}

We compare the total computational work (the number of floating-point operations, where we count a pair of addition and multiplication as one floating point operation) for the greedy Kaczmarz inner-iteration preconditioning using Algorithm~\ref{The Greedy Kaczmarz Method} (GK), with GK modified by precomputing and storing $C=AA^\mathsf{T}$ once beforehand and updating the residual vector $\boldsymbol{s}$ using equation~\eqref{r} (MGK).

Let $A$ be an $m \times n$ matrix. Assume that the number of outer iterations of the F-AB-GMRES is $k$, and that the number of inner Kaczmarz iterations is fixed at $\ell$ for each outer iteration.

First, consider the case when $A$ is dense. Then, the total work for GK is given by
\begin{equation*}
w_{\text{GK}}^\text{d} = k \ell (mn+n+m).
\end{equation*}
The first term corresponds to step 3, the second to step 4 of Algorithm~\ref{The Greedy Kaczmarz Method}, and the third to the residual norm computation in step 4 of Algorithm~\ref{AGPK}, respectively.
The total work for MGK is
\begin{equation*}
w_{\text{MGK}}^\text{d} = m^2 n + k \ell (2m+n).
\end{equation*}
The first term is for computing  $C=AA^\mathsf{T}$ once beforehand, the second for the update in step 3 of Algorithm~\ref{The Greedy Kaczmarz Method} using equation~\eqref{r} and step 4 of Algorithm~\ref{AGPK}, and the third for step 4 of Algorithm~\ref{The Greedy Kaczmarz Method}, respectively.
Hence,
\begin{equation*}
w_{\text{GK}}^\text{d} - w_{\text{MGK}}^\text{d} = m \left[ k \ell (n-1) - mn \right].
\end{equation*}
Therefore, 
\begin{equation}\label{wd}
w_{\text{MGK}}^\text{d}  < w_{\text{GK}}^\text{d} \quad \Longleftrightarrow \quad k \ell > m \left( 1 + \frac{1}{n-1} \right). 
\end{equation}

Next, consider the case when $A$ is sparse and the position of the nonzero elements are random. Let $\mathrm{nz}$ be the number of nonzero elements of $A$. Define ${\displaystyle q= \mathrm{nz}/m}$ as the average number of nonzero elements per row of $A$. Thus, the density of $A$ is $d={\displaystyle q/n}$. Assume that the computational work to compute $C = (c_{ij}) = AA^\mathsf{T}$ is $m^2q$.
Let the density of $C$ be $p$.

Then, the total work for GK is
\begin{equation*}
w_{\text{GK}}^\text{s} = k \ell (\mathrm{nz} + q +m) = k \ell (qm+q+m).
\end{equation*}
The first term is for step 3, the second for step 4 of Algorithm~\ref{The Greedy Kaczmarz Method}, and the third for step 4 of Algorithm~\ref{AGPK}, respectively.
The total work for MGK is
\begin{equation*}
w_{\text{MGK}}^\text{s} = m^2q+ k \ell \left(q+mp+m\right).
\end{equation*}
The first term is for computing $AA^\mathsf{T}$, the second for step 4 of Algorithm~\ref{The Greedy Kaczmarz Method}, the third for the update in equation~\eqref{r}, and the fourth for step 4 of Algorithm~\ref{AGPK}, respectively.  
Hence,
\begin{equation*}
w_{\text{GK}}^\text{s} - w_{\text{MGK}}^\text{s} = m\left[k \ell \left(q-p\right)  -mq\right].
\end{equation*}
Therefore,
\begin{equation}\label{ws}
w_{\text{MGK}}^\text{s}  <  w_{\text{GK}}^\text{s} \quad \Longleftrightarrow \quad k \ell > m\left(1 + \frac{p}{q-p}\right).
\end{equation}

The density $p$ of $C=AA^\mathsf{T}$ can be estimated as follows. The probability that $c_{ij} \ne 0$ for $i \ne j$ is ${p_{\text{nd}}} = 1 - {\left(1 - {d^2}\right)^n}$, and the probability that $c_{ii} \ne 0$ is ${p_\text{d}} = 1 - {(1 - d)^n}$. 
Here, the estimation of the probability ${p_{\text{nd}}}$ is based only on the probability (density) $d=q/n$ of an element of $A$ being nonzero, and not on its numerical value, so we have not taken into account the case when the rows of $A$ are orthogonal, which can be considered to be generically negligible.

Therefore the probability that $c_{ij} \ne 0$ (or the density of $C$) is
\begin{align} \label{p}
p &= \frac{{\left( {{m^2} - m} \right){p_{\text{nd}}} + m{p_\text{d}}}}{{{m^2}}} \nonumber\\
&= 1 - \left( {1 - \frac{1}{m}} \right){\left( {1 - {d^2}} \right)^n} - \frac{1}{m}{\left( {1 - d} \right)^n}.
\end{align}
If $d=1$ ($A$ is dense),~\eqref{p} implies $p=1$. Then, also $q=n$, so that~\eqref{ws} agrees with~\eqref{wd}.
If $d \ll 1$, $p$ can be approximated as 
\begin{equation*}
p \sim 1 - \left( {1 - \frac{1}{m}} \right){e^{ - n{d^2}}} - \frac{1}{m}{e^{ - q}}.
\end{equation*}

Table~\ref{tab:6} gives estimated (using~\eqref{p}) and actual values of the density of $C=AA^\mathsf{T}$ for the matrices used in our experiments.
The estimation captures the trend of the actual density qualitatively.

\begin{table}[!t]
	\caption{Estimated and actual densities of the matrix $C=AA^\mathsf{T}$.}\label{tab:6}
	\begin{center}
		\begin{tabular}{rrrrrr}
			\hline
			\multicolumn{1}{c}{matrix} & \multicolumn{1}{c}{$m$} & \multicolumn{1}{c}{$n$} &\multicolumn{1}{c}{$d$} & \multicolumn{1}{c}{$p(\text{estimated})$} & \multicolumn{1}{c}{$p(\text{actual})$} \\
			\hline
			RANDL1 & 5000 & 500 & 0.2 & 1.00 & 0.745 \\
			RANDL2 & 5000 & 500 & 0.2 & 1.00 & 0.730\\
			RANDL3 & 5000 & 500 & 0.2 & 1.00 & 0.731\\
			RANDL4 & 5000 & 500 & 0.2 & 1.00 & 0.714\\
			RANDL5 & 5000 & 500 & 0.2 & 1.00 & 0.719\\
			RANDL6 & 5000 & 500 & 0.2 & 1.00 & 0.719\\
			illc1850&1850&712& 6.60$\times 10^{-3}$&0.0307&0.153\\
			gen&2561&769& 3.20$\times 10^{-2}$&0.546&0.556\\
			photogrammetry2&4472&936& 8.90$\times 10^{-3}$&0.0709&0.0321\\
			Maragal$\_$3 & 1682 & 858 & 1.27$\times 10^{-2}$ & 0.131& 0.160\\
			Maragal$\_$4 & 1964 & 1027 & 1.32$\times 10^{-2}$ & 0.165 & 0.126\\
			Maragal$\_$5 & 4654& 3296 & 6.10$\times 10^{-3}$ & 0.115 &  0.0728\\
			RANDL1T  & 500 & 5000 & 0.2 & 1.00 & 0.940 \\
			RANDL2T  & 500 & 5000 & 0.2 & 1.00 & 0.950\\
			RANDL3T  & 500 & 5000 & 0.2 & 1.00 & 0.926\\
			RANDL4T  & 500 & 5000 & 0.2 & 1.00 & 0.927\\
			RANDL5T  & 500 & 5000 & 0.2 & 1.00 & 0.932\\
			RANDL6T  & 500 & 5000 & 0.2 & 1.00 & 0.932\\
			Maragal$\_$3T  & 858 & 1682 & 1.27$\times 10^{-2}$ & 0.240 & 0.562\\
			Maragal$\_$4T  & 1027 & 1964 & 1.32$\times 10^{-2}$ & 0.292 & 0.669\\
			Maragal$\_$5T & 3296 & 4654 & 6.10$\times 10^{-3}$ & 0.158 & 0.461\\
			illc1850T&712&1850& 6.60$\times 10^{-3}$&0.0777&0.0179\\
			genT&769&2561& 3.20$\times 10^{-2}$&0.928&0.646\\
			photogrammetry2T&936&4472& 8.90$\times 10^{-3}$&0.296&0.0944\\
			\hline
		\end{tabular}
		\begin{tablenotes}
			$m$: number of rows of $A$, $n$: number of columns of $A$, 
			$d$: density of nonzero elements of $A$, $p$: density of nonzero elements of $C=AA^\mathsf{T}$.
		\end{tablenotes}
	\end{center}
\end{table}

As for the CPU time, the computation of $AA^\mathsf{T}$ should perform relatively more efficiently than the flops count suggests, especially for the dense case, due to fast memory access.

\section*{Appendix B} \textbf{Equivalence between \eqref{l2} and \eqref{n2}}.

Necessity:
Let $f\left( \boldsymbol{x} \right) = \frac{1}{2}\left\| \boldsymbol{x} \right\|_2^2 + {\boldsymbol{\lambda} ^\mathsf{T}}\left( {\boldsymbol{b} - A\boldsymbol{x}} \right)$ be the Lagrange function, where $\boldsymbol{\lambda}\in {\mathbb{R}^m}$ is the Lagrange multiplier.
Since
\begin{align*}
\frac{{\partial f\left( \boldsymbol{x} \right)}}{{\partial {x_i}}} &= {x_i} - \boldsymbol{a} _i^\mathsf{T}\boldsymbol{\lambda}, \quad i=1,2,\ldots,n, \\
\frac{{\partial f\left( \boldsymbol{x} \right)}}{{\partial {\lambda _i}}} &= {b_i} - \boldsymbol{\alpha} _i^\mathsf{T}\boldsymbol{x}, \quad i=1,2,\ldots,m,
\end{align*}
where $\boldsymbol{a}_i$ is the $i$th column of $A$ and $\boldsymbol{\alpha} _i^\mathsf{T}$ is the $i$th row of $A$, we have
\begin{align*}
\frac{{\partial f\left( \boldsymbol{x} \right)}}{{\partial \boldsymbol{x}}} &= 0 \quad \Longleftrightarrow \quad \boldsymbol{x} = {A^\mathsf{T}}\boldsymbol{\lambda}, \\
\frac{{\partial f\left( \boldsymbol{x} \right)}}{{\partial \boldsymbol{\lambda} }} &= 0 \quad \Longleftrightarrow \quad A\boldsymbol{x} = \boldsymbol{b}.
\end{align*}
Hence, the solution of \eqref{l2} satisfies \eqref{n2}.

Sufficiency:
Let $\boldsymbol{x}=\boldsymbol{x}_1+\boldsymbol{x}_2$, where ${\boldsymbol{x}_1} \in \mathcal{N}{(A)^ \bot } = \mathcal{R}({A^\mathsf{T}})$ and ${\boldsymbol{x}_2} \in \mathcal{N}(A)$.
Let $\boldsymbol{x}_1={A^\mathsf{T}}\boldsymbol{u}$. Then, we have $A{\boldsymbol{x}_1} = A\boldsymbol{x} = \boldsymbol{b}$.
If $\boldsymbol{y}\in {\mathbb{R}^n}$ satisfies $A\boldsymbol{y}=\boldsymbol{b}$, then we have $A(\boldsymbol{y}-\boldsymbol{x}_1)=0$. Let $\boldsymbol{y}-\boldsymbol{x}_1=\boldsymbol{t} \in \mathcal{N}(A)$. Then, we have $\boldsymbol{y}=\boldsymbol{x}_1+\boldsymbol{t}$, where ${\boldsymbol{x}_1} \in \mathcal{N}{(A)^ \bot }$ and $\boldsymbol{t} \in \mathcal{N}(A)$. Since 
$${\left\| \boldsymbol{y} \right\|^2} = {\left\| {{\boldsymbol{x}_1}} \right\|^2} + {\left\| \boldsymbol{t} \right\|^2} \ge {\left\| {{\boldsymbol{x}_1}} \right\|^2},$$
there exists ${\boldsymbol{x}_1} = {A^\mathsf{T}}\boldsymbol{u}$, where ${\boldsymbol{x}_1} \in \arg \{ {\min {{\left\| \boldsymbol{x} \right\|}^2}, A\boldsymbol{x} = \boldsymbol{b}} \}$.

If ${\boldsymbol{x}_1} = {A^\mathsf{T}}{\boldsymbol{u}_1}$, $A\boldsymbol{x}_1=\boldsymbol{b}$, and ${\boldsymbol{x}_2} = {A^\mathsf{T}}{\boldsymbol{u}_2}$, $A\boldsymbol{x}_2=\boldsymbol{b}$, we have ${\boldsymbol{x}_2} - {\boldsymbol{x}_1} \in \mathcal{R}({A^\mathsf{T}}) = \mathcal{N}{(A)^ \bot }$ and $A(\boldsymbol{x}_2-\boldsymbol{x}_1)=0$, we have ${\boldsymbol{x}_2} - {\boldsymbol{x}_1} \in \mathcal{N}(A)$. Hence, ${\boldsymbol{x}_2} - {\boldsymbol{x}_1} \in \mathcal{N}{(A)^ \bot } \cap \mathcal{N}(A) = \{{\boldsymbol {0}}\}$. Thus, $\boldsymbol{x}_1=\boldsymbol{x}_2.$

\section*{Acknowledgement}
We would like to thank Professor Zhong-Zhi Bai for stimulating discussions and valuable remarks. 
We would also like to thank the referees for their valuable comments.

\bibliographystyle{siamplain}
\bibliography{references}

\end{document}